\documentclass[preprint,twoside,12pt]{amsart}

\usepackage{amssymb,color,enumerate,amscd,graphicx}
\usepackage{amsmath,amsfonts,amsopn,amsthm}
\usepackage{subfigure,float}
\usepackage{pdfsync}
\usepackage[top=2cm,bottom=4cm,left=2.0cm,right=2cm,asymmetric]{geometry}

\newtheorem{theorem}{Theorem}
\newtheorem{lemma}[theorem]{Lemma}

\newtheorem{prop}[theorem]{Proposition}
\theoremstyle{definition}
\newtheorem{definition}[theorem]{Definition}

\newtheorem{ass}{Assumption}
\theoremstyle{remark}
\newtheorem{remark}{Remark}[section]

\numberwithin{equation}{section}

\newcommand{\new}{\newcommand}

\providecommand{\nor}[1]{\lVert{#1}\rVert}
\providecommand{\abs}[1]{\lvert{#1}\rvert}
\providecommand{\set}[1]{\{#1\}}
\providecommand{\scal}[2]{\left\langle{#1},{#2}\right\rangle}

\providecommand{\det}[1]{\operatorname{det}(#1)}

\providecommand{\dim}[1]{\operatorname{dim}(#1)}

\new{\R}{\mathbb R}
\new{\C}{\mathbb C}
\new{\N}{\mathbb N}
\new{\Z}{\mathbb Z}

\new{\hh}{\mathcal H}
\new{\kk}{\mathcal K}

\new{\la}{\lambda}
\new{\eps}{\epsilon}

\newcommand{\intl}{\int\limits}

\def\Spdr{Sp(d,\R)}
\def\Symdr{\rm Sym(d,\R)}
\def\Gldr{{\rm GL(d,\R)}}

\begin{document}

\title[Admissible
  vectors.]{Reproducing subgroups  of Sp(2,R).\\ Part II:  admissible
  vectors.}

\author{G.~Alberti}
\address{G. Alberti, Mathematical Institute\\ 24-29 St Giles'  Oxford, OX1 3LB, UK}
\email{Giovanni.Alberti@maths.ox.ac.uk }
\author{F.~De~Mari}
\address{F. De Mari, DIMA\\
Via Dodecaneso, 35, 16146 Genova, Italy}
\email{demari@dima.unige.it}
\author{E.~De Vito}
\address{E. De Vito, DIMA\\Via Dodecaneso, 35
\\16146 Genova, Italy}
\email{devito@dima.unige.it}
\author{L.~Mantovani}
\address{L. Mantovani, DIMA\\Via Dodecaneso, 35
\\16146 Genova, Italy}
\email{mantovani@dima.unige.it}

\begin{abstract} In part~I we introduced the class ${\mathcal E}_2$ of
  Lie subgroups of $Sp(2,\R)$ and obtained a classification up to
  conjugation (Theorem~1.1). Here, we determine for which of these
  groups the restriction of the metaplectic representation gives rise
  to a reproducing formula. In all the positive cases we characterize
  the admissible vectors with a generalized Calder\'on equation. They
  include products of 1D-wavelets, directional wavelets, shearlets,
  and many new examples.
\end{abstract}

\subjclass[2010]{Primary: 42C40,43A32, 43A65 }

\maketitle

\section{Introduction}

In this paper we complete the program begun in Part~I, where we introduced the class  ${\mathcal E}_2$ of Lie subgroups of $Sp(2,\R)$, consisting of semidirect products $\Sigma\rtimes H$, where the normal factor $\Sigma$ is a vector space of $2\times2$ symmetric matrices and $H$ is a connected Lie subgroup of $GL(2,\R)$. We then obtained a classification up to $Sp(2,\R)$-conjugation, which is restated below in Theorem~\ref{THELIST} in more explicit notation. Here we prove  which of the groups in the class ${\mathcal E}_2$ are reproducing,
namely for which of them there exists an admissible vector $\eta\in L^2(\R^2)$ such that
  \begin{equation}
\nor{f}^2= \intl_G\abs{\scal{f}{U_g\eta }}^2\,dg 
\label{eq:2}
  \end{equation}
holds for every $f\in L^2(\R^2)$, where $dg$ is a fixed left Haar measure on $G$, and $U$ is the restriction to $G$ of the metaplectic representation. This is the content of Proposition~\ref{prop:not_repr} and of Theorem~\ref{main}.  

Secondly, and most importantly,  we want to describe all the admissible vectors in the reproducing cases. Both problems are solved by means of the general theorems proved in \cite{dede10}. We also prove an auxiliary result concerning what we call {\it orbit equivalence}, a notion designed in order to find conditions under which two groups that are {\it not} conjugate within $Sp(2,\R)$ (hence  not equivalent according to the  general representation-theoretic principles that are relevant in this context) do exhibit the same analytic features, in the sense that they have coinciding sets of admissible vectors. Orbit equivalence, which is an analytic condition, allows us to treat several families in a concise way. In some sense, it should  be thought of as a version of a change of variables. It is worthwile observing that all but one families depending on a parameter (precisely: those listed as (2D.1), (2D.2), (2D.3), (3D.4) together with (3D.5), (3D.6) and (3D.7)) are shown to consist of orbitally equivalent groups. The exception is notably the family (4D.4) which is the family of shearlet groups, indexed by the parameter $\alpha\in(-1,0]$ (the case $\alpha=0$ corresponding to the usual shearlets with parabolic scaling), for which we show that no orbit equivalence can possibly exist.

 In the end, the most interesting new phenomenon is perhaps the appearence of four non-conjugate classes of two-dimensional groups. Each of them is isomorphic to the affine group, but the metaplectic representation restricted to each of them is of course highly reducible and clearly not equivalent to the standard wavelet representation, the most obvious difference being that the former analyses two-dimensional signals and the latter acts on $L^2(\R)$.
 
 A broader discussion concerning the r\^ole of the metaplectic representation and of the groups in  ${\mathcal E}_2$ in signal analysis is to be found in Part~I. Other significant contributions in this circle of ideas are
 contained in the recent papers \cite{kingczaja12} and \cite{kampanat12}.

The paper is organised as follows. In the rather long Section~\ref{notation}, after introducing some notation, we review the results contained in \cite{dede10} in some detail. In Section~\ref{orbitequivalence} we discuss orbit maps and introduce the notion of orbit equivalence, which will be used later. Here the main result is 
Theorem~\ref{OE}. Section~\ref{AV} contains the classification result and the equations for the admissible vectors.

\section{Notation and known results}\label{notation}

In this section we recall the main results concerning the 
voice transform associated to the metaplectic representation restricted to a
suitable class of Lie subgroups of $Sp(2,\R)$. We review rather thoroughly the results contained in \cite{dede10} because we need a much simpler formulation here, which would not be  plain to infer from \cite{dede10} at first reading.

\subsection{Notation} Given a locally
compact second countable space $X$, $C_c(X)$ is the
space of continuous complex functions with compact support. A Radon
measure $\nu$ on $X$ is a positive
measure defined on the Borel $\sigma$-algebra of $X$ and finite on compact
subsets, and $L^2(X,\nu)$ is the Hilbert space of the complex
functions on $X$ which are  square
integrable  with respect to $\nu$. The  norm and the scalar product of $L^2(X,\nu)$
are denoted by
$\nor{\cdot}_{\nu}$ and $\scal{\cdot}{\cdot}_{\nu}$, respectively, and  if
no confusion arises the dependence on $\nu$ is omitted.

\subsection{Reproducing groups of $\Spdr$.}
Given an integer $d\geq 1$, we denote by $\Spdr$ the symplectic group acting on $\R^{2d}$ and by $\mu$ the metaplectic
representation of $\Spdr$ (see for example \cite{fol89}). We recall that $\mu$ is a projective, unitary, strongly continuous representation of $\Spdr$ acting on $L^2(\R^d,dx)$,  where $dx$ is
the Lebesgue measure of $\R^d$. Given a Lie subgroup $G$ of $\Spdr$, we denote by $dg$ a left Haar
measure of $G$, by $\Delta_G$ its modular function, and by $U$ the restriction of $\mu$ to $G$.

Fix now a Lie subgroup $G$ of $\Spdr$. Regarded the space $L^2(\R^d,dx)$ as the set
of \mbox{$d$-dimensional} signals, we  introduce the voice transform associated to
the representation $U$ of  $G$.  Recall that a voice transform is obtained by
choosing an analysing function $\eta\in L^2(\R^d,dx)$, and then
defining for any signal $f\in L^2(\R^d,dx)$  the continuous function 
\[ 
V_\eta f:G\to\C,\qquad V_\eta f(g)=\scal{f}{U_g\eta}.
\]
For an arbitrary $G$ it may well happen that  $V_\eta f$ is not in
$L^2(G,dg)$. The following definition  selects the family of subgroups for
which the voice transform becomes an isometry.
Notice that we do not require $U$ to be irreducible.
\begin{definition}
A Lie subgroup $G$ of $\Spdr$ is {\em a reproducing group} if
  there exists $\eta\in L^2(\R^d,dx)$ such that formula \eqref{eq:2} holds  for every $f\in L^2(\R^d,dx)$.
The function $\eta$ is called an {\em admissible
    vector} for $G$.
\end{definition}
Under these circumstances, $V_\eta f$ is in $L^2(G,dg)$ and the  analysis operator
$f\mapsto V_\eta f$ is an isometry from $L^2(\R^d,dx)$ into
$L^2(G,dg)$ intertwining $U$ with the left regular
representation of $G$.  Furthermore,
the following reproducing formula holds true  for
all $f\in L^2(\R^d,dx)$
\begin{equation}
  \label{eq:1}
f =\intl_ G \scal{f}{U_g\eta}\, U_g\eta \ dg,
\end{equation}
where the integral must be interpreted in the weak sense.

\begin{remark}\label{conjugate} Take  a reproducing group $G$  and $g\in\Spdr$.  It is 
well-known that the conjugate group $gGg^{-1}$ is
  reproducing, too.  Furthermore, if $\eta\in L^2(\R^d,dx)$ is
  an admissible vector for $G$, then $\mu_g\eta$ is an admissible
  vector for $gGg^{-1}$.
\end{remark}

\subsection{The class $\mathcal E$} In this paper, we consider subgroups of $\Spdr$ whose elements are
``triangular'' $d\times d$-block matrices.  For a discussion on this choice the reader is referred to Part~I. 
\begin{definition}
A Lie subgroup $G$ of $\Spdr$ belongs to the class $\mathcal E$  if
it is of the form
  \begin{equation*}
    G=\Bigl\{\begin{bmatrix}h & 0 \\\sigma h& ^th^{-1} \end{bmatrix}: h\in
    H,\ \sigma\in \Sigma\Bigr\},
  \end{equation*}
  where $H$ is a connected Lie subgroup of $\Gldr$ and $\Sigma$ is a  subspace of the space $\Symdr$ of 
  $d\times d$ symmetric matrices. We further require that both $\Sigma$ and $H$ are not trivial.
Whenever needed, we write $\mathcal E_d$ to specify the size.\end{definition}
  
  In order for $G$ to be a group it is necessary and sufficient that 
  $h^\dag[\Sigma]=\Sigma$ for all  $h\in H$, where
  \begin{equation}
h^\dag[\sigma]={^th}^{-1}\sigma h^{-1}.
\label{dagger1}
\end{equation}
If $G\in\mathcal E$, both $\Sigma$ and $H$ are naturally identified as Lie
subgroups of $G$. Clearly,  $\Sigma H=G$, $\Sigma\cap H=\set{e}$, $\Sigma$ is
a normal subgroup of $G$ and it is invariant under the action
of $H$ given by~\eqref{dagger1}, so that $G$ is the semi-direct product $\Sigma\rtimes H$.

 For the remaining part of this section, we fix a group $G=\Sigma\rtimes H$ in the class
 $\mathcal E$.
A left Haar measure and the modular function of $G$ are
\begin{equation}\label{eq:deltaG}
dg= \chi(h)^{-1}d\sigma dh\qquad
\Delta_G(\sigma,h)=\chi(h)^{-1}\Delta_H(h),
\end{equation}
where $d\sigma$ is a Haar measure of $\Sigma$, $dh$ is a left Haar
measure of $H$, $\Delta_H$ is the modular function of $H$, and $\chi$ is the positive character of $H$ given, for all $h\in H$, by
\begin{equation}\label{eq:alpha}
\chi(h)=\abs{\det{\sigma\mapsto h^\dag[\sigma]}}.
\end{equation} 
We denote by $\Sigma^*$ the dual of $\Sigma$. The contragredient action of \eqref{dagger1} is then given by
\begin{equation}
h[\sigma^*](\sigma)=\sigma^*(({h^{-1}})^{\dagger}[\sigma])=\sigma^*(\,^th\sigma h),
\qquad
\sigma^*\in\Sigma^*, \sigma\in\Sigma, h\in H.
\label{semidirect}
\end{equation}
Since the action \eqref{semidirect} will play a more relevant r\^ole than the action \eqref{dagger1},
we have chosen the simpler notation for the former.
\subsection{The representation $U$} The
restriction $U$ of the metaplectic representation to $G\in\mathcal E$ is completely
characterized by a ``symbol'' $\Phi$, as we now explain.
In its standard form, $U$ acts on $f\in L^2(\R^d)$ by
\begin{equation}
U_{(\sigma,h)}f(x) 
= \beta(h)^{-\frac12}\,e^{\pi i\scal{\sigma x}{x}}f(h^{-1}x),\qquad\text{a.e. }x\in\R^d,
\end{equation}
where $\beta$ is the positive character  of $H$ given by
\[
\beta(h)=\abs{\det{h}}.
\]
Now, given $x\in\R^d$,  the map $\sigma\mapsto -\frac{1}{2}\,\scal{\sigma
  x}{x}$ is a linear functional on $\Sigma$ and hence it defines a unique element
$\Phi(x)\in\Sigma^*$ by the requirement that for all $\sigma\in\Sigma$
\begin{equation}\label{eq:defPhi}
{\Phi(x)}(\sigma)=-\frac{1}{2}\,\scal{\sigma x}{x}_{\R^d}.
\end{equation}
The corresponding function $\Phi: \R^d\to \Sigma^*$  has a fundamental
invariance property. Indeed, noting that the group $H$ acts naturally
on $\R^d$ by means of
\[ h.x=hx\qquad x\in\R^d,\ h\in H,\]
it follows that for all $x\in \R^d$ and $h\in H$,
\begin{equation}
\Phi(h.x)=h[\Phi(x)].\label{PHI}
\end{equation}
This is seen by observing that for all $\sigma\in\Sigma$ we have
\[  
\Phi(h.x)(\sigma)
  =-\frac{1}{2}\,\scal{{^th}\sigma h x}{x}_{\R^d} 
  =\Phi(x)((h^{-1})^\dagger[\sigma])
  = h[\Phi(x)](\sigma).
\]
Therefore, for  $\sigma\in \Sigma$ and  $h\in H$, by
\begin{equation*}
U_{(\sigma,h)}f(x) 
 =\beta(h)^{-\frac12}\,e^{-2\pi i \Phi(x)(\sigma)}f(h^{-1}x),
\end{equation*}
which exhibits $U$ as a representation of the kind considered in \cite{dede10}, with the quadratic symbol~$\Phi$.

\subsection{More notation} In the examples, we often parametrise the $n$-dimensional
  vector space $\Sigma$ by selecting a  basis $\{\sigma_1,\ldots,\sigma_n\}$. Clearly
  \begin{equation*}
  \Sigma=\set{\sigma(u):=u_1\sigma_1+\ldots+u_n\sigma_n\,:\, u=(u_1,\ldots,u_n)\in\R^n}.  
  \end{equation*}
With this choice, we fix the Haar measure $d\sigma$ on $\Sigma$ as the push-forward
of the Lebesgue measure under the linear isomorphism $u\mapsto
\sigma(u)$. By means of this choice we also identify $\Sigma^*$ with $\R^n$, that is, if 
 $\{\sigma_1^*,\ldots,\sigma_n^*\}$ is the dual basis, then we map 
 \[
 y\mapsto \sigma^*(y):=y_1\sigma_1^*+\ldots+y_n\sigma_n^*.
\]
The Haar measure $d\sigma^*$ on $\Sigma^*$ will be the push-forward
of the Lebesgue measure under the linear isomorphism $y\mapsto\sigma^*(y)$.
Furthermore, for $h\in H$ the action~\eqref{semidirect}
defines the matrix $M_h\in\operatorname{GL}(n,\R)$  in the chosen basis, that is 
\[
h^\dag[\sigma(u)]= \sigma(M_hu)\qquad\iff\qquad
h^\dag[\sigma_j]=\sum_i \sigma_i (M_h)_{ij}\quad j=1,\ldots,n,
\]
and the dual action~\eqref{semidirect} becomes the contragredient of
$M_h$, namely
\[
h[\sigma^*(y)]=\sigma^*(^t\!M^{-1}_hy).
\]
The quadratic map $\Phi$ given by~\eqref{eq:defPhi} defines
$n$-smooth maps $\varphi_1,\ldots,\varphi_n:\R^d\to\R$ by
\[
\Phi(x)(\sigma(u))=-\frac12\scal{\sigma(u)x}{x}_{\R^d}=u_1\varphi_1(x)+\ldots+u_n\varphi_n(x)
\quad\iff\quad 
\varphi_i(x)=-\frac12\scal{\sigma_1x}{x}_{\R^d}.
\]
 In the following, with slight abuse of notation, we
write
\begin{equation}
h^\dag[u]:=M_hu,
\qquad
h[y]:={}^t\!M^{-1}_h y.
\label{TYaction}
\end{equation}
Finally, we denote by $D\Phi(x)$ the $n\times d$ Jacobian matrix 
\[
D\Phi(x)=\begin{smallmatrix}
  \dfrac{\partial(\varphi_1,\ldots,\varphi_n)}{\partial(x_1,\ldots,x_d)}
\end{smallmatrix}
\]
and by $J\Phi(x)= \sqrt{\det D\Phi(x)\,^tD\Phi(x)}$ the Jacobian determinant. 
With this notation, the critical points of $\Phi$ are precisely the solutions of
the equation $J\Phi(x)=0$. The set of critical points
is invariant under any change of basis in $\Sigma$, whereas the Haar
measure $d\sigma$ and $J\Phi$ will change by a positive constant, making the
condition~\eqref{eq:1} on the admissible vectors invariant.

In the next two sections we recall the main results of \cite{dede10}  applied to our setting.

\subsection{Geometric characterization}
In what follows, we need a technical assumption that  is necessary 
in order to avoid pathological phenomena and it is
verified in all our examples.  To state it, we identify $\R^n$with $\Sigma^*$ under
the map $y\leftrightarrow\sigma^*(y)$ and hence we regard $\R^n$
 as an $H$-space with respect to the second
action in~\eqref{TYaction}. For any $y\in\R^n$, we denote by
$H[y]=\set{h[y]\in{\R^n}: h\in H}$ the
corresponding orbit and by $H_y=\set{h\in H: h[y]=y}$ the stability subgroup.
The following assumption will be made throughout the remaining sections.
\begin{ass}\label{H1}
For every $y\in{\R^n}$, the orbit  $H[y]$ is locally closed in ${\R^n}$.
\end{ass}

The first result gives some necessary conditions in order that $G$ is
a reproducing group. These conditions become sufficient if the stabilizers are compact.
\begin{theorem}\label{th:suffnec}
Take $G=\Sigma\rtimes H\in\mathcal E$.
If  $G$ is a reproducing group, then
 \begin{enumerate}[i)]
\item\label{2} $G$ is non-unimodular;
\item\label{1} $\dim{\Sigma}\leq d$;
\item\label{3} the set of critical points of $\Phi$, which is an $H$-invariant
  closed subset of $\R^d$, has zero Lebesgue measure.
\end{enumerate}
Furthermore, if $\dim{\Sigma}=d$, then
\begin{enumerate}[iv)]
\item for almost every $y\in \Phi(\R^d)$ the stability
  subgroup $H_y$ is compact.
\end{enumerate}
Conversely, if  \ref{2}), \ref{1})  \ref{3}) and iv) 
(without assuming $\dim{\Sigma}=d$) hold true, then  $G$
is reproducing.
\end{theorem}
\begin{proof} Hereafter, all the cited results refer
  to~\cite{dede10}.  Assume that $G$ is reproducing. Theorem~1 together with Lemma~2 
implies~\ref{1}) and \ref{3}). Since $\Phi$ is a
 quadratic map and the action $x\mapsto h.x$ is linear, Proposition~5
 gives~ \ref{2}). Finally, if $\dim{\Sigma}= d$, Theorem~10 yields~iv).

Conversely, Theorem~9, where  \ref{1}) is
understood, proves that $G$ is a reproducing group.
\end{proof}
\subsection{Analytic conditions}\label{sub:analytic}
Because of Theorem~\ref{th:suffnec}, in this section we assume the existence of an open $H$-invariant
subset $X\subset\R^d$ with negligible complement, whence $L^2(X,dx)=L^2(\R^d,dx)$, such that the Jacobian of $\Phi$ is strictly
positive on $X$. As a consequence, $Y:=\Phi(X)$ is an \mbox{$H$-invariant}
open set of $\R^n$, for every $y\in Y$ the level set  $\Phi^{-1}(y)$ is
a Riemannian submanifold (with Riemannian measure $dv_y$) and the
Lebesgue measure $dy$ of $\R^n$, restricted to $Y$, is a
pseudo-image measure of the Lebesgue measure $dx$ of $X$ (see Lemma~2
in \cite{dede10}).

The next result is based on classical disintegration
formul\ae \ of measures. We refer to \cite{bourbaki_int} for the general
theory. A short account is given in Appendix~A of \cite{dede10}.

The coarea formula \cite{Federer69} yields the  disintegration \eqref{eq:12} of the Lebesgue measure on $X$.
\begin{prop}\label{prop:dis1}
  There exists a unique family $\{\nu_y\}_{y\in Y} $ of Radon
  measures on $X$ such that:
  \begin{enumerate}[a)]
  \item for every $y\in Y$ the measure $\nu_y$ is concentrated
    on $\Phi^{-1}(y)$;
  \item for every $\varphi\in C_c(X)$
    \begin{equation}
      \int_X
      \varphi(x)d\nu_y(x)=\int_{\Phi^{-1}(y)}\varphi(x)  \frac{
        dv_y(x)}{(J\Phi)(x)};\label{eq:11}
    \end{equation}
  \item for every $\varphi\in C_c(X)$
    \begin{equation}
      \int_X
      \varphi(x)dx=\int_Y\Bigl(\int_X
      \varphi(x)d\nu_y(x)\Bigr)dy;\label{eq:12}
    \end{equation}
  \item for every $\varphi\in C_c(X)$ and $h\in H$
    \begin{equation}
      \int_X
      \varphi(h^{-1}.x)d\nu_{h[y]}(x)=\chi(h)\beta(h)\int_X\varphi(x)d\nu_y(x).\label{eq:13}
    \end{equation}
  \end{enumerate}
\end{prop}
The general theory of disintegration of measures gives that
\begin{equation}
L^2(X,dx) =\int_Y  L^2(X,\nu_y) dy,
\label{eq:6}
\end{equation}
where the direct  integral  of the family of Hilbert spaces
$\{L^2(X,\nu_y) \}_{y\in Y}$ is taken with respect to the
measurable structure defined by $C_c(X)\subset L^2(X,\nu_y)$ (see
\cite{dix57}  for the  theory of direct integrals and Proposition~6 in \cite{dede10} for \eqref{eq:6}). Accordingly,
for any $f\in L^2(X,dx)$, we write 
\[f=\int_Y  f_y  dy,
\qquad
\nor{f}^2=\int_Y \nor{f_y}^2 dy\]
where $f_y\in L^2(X,\nu_y)$. Since $\nu_y$ is
concentrated on $\Phi^{-1}(y)$,  $f_y$ can be regarded as
a function on $\Phi^{-1}(y)$. In particular, if $f$ is
continuous, then $f_y$ is the restriction of $f$ to
$\Phi^{-1}(y)$.

To state the next disintegration formula, we fix a locally compact
second countable topological space $Z$,  two Borel maps
 $\pi:Y\to Z$ and  $o:Z\to Y$, and a Radon measure $\lambda$ on $Z$ such that:
  \begin{enumerate}[i)]
  \item $\pi(y)=\pi(y')$ if and only if $y$ and
    $y'$ belong to the same orbit;
\item  for all $z\in \pi(Y)$ we have $\pi(o(z))=z$;
\item\label{pseudo} a set $E\subset Z$ is $\lambda$-negligible if and only if $\pi^{-1}(E)$ is
Lebesgue negligible.
\end{enumerate}
\begin{remark}
The set $Z$ is a sort of redundant parametrization of the orbit space $Y/H$ by means
of a  locally compact space (in general $Y/H$ is not Hausdorff, hence not locally compact, with respect
to the quotient topology), $\pi$ replaces the canonical projection
from $Y$ onto $Y/H$ (though, in general, $\pi$ is not surjective) and
$o$ is a Borel section for $\pi$.  Property~\ref{pseudo}) states that
the $\lambda$-negligible sets of $Z$ are uniquely defined by $\pi$ and
the Lebesgue measure of $Y$, so that the measure class of $\lambda$ is
unique.  As a consequence of~\ref{pseudo}), $\lambda$ is concentrated
on $\pi(Y)$. 
The existence of $Z$, $\lambda$, $\pi$ and $o$ follows from
 Assumption~\ref{H1} and Theorem~2.9 in \cite{effros65},
as explained in \cite{dede10}  before Theorem~3. In many examples, however, $Y/H$ is
indeed locally compact and we may safely assume that $Z=Y/H$ and that $\pi$ is the canonical projection.
As for the measure $\lambda$, one takes an $L^1$ positive density $f$ on $Y$ and then builds the image measure of $f\cdot dy$ under $\pi$.

\end{remark}
The next result is again obtained from  the theory of disintegration
of measures (see the proof of Theorem~4 of \cite{dede10} for a list of
references).

\vskip1cm

\begin{prop}\label{prop:dis2}
There is a family of
  Radon measures $\set{\tau_z}_{z\in Z}$ on $Y$ with the following
  properties:
\begin{enumerate}[a)]
  \item for every $z\in Z$, the measure $\tau_z$ is concentrated on the
    orbit $\pi^{-1}(z)$
    and, for every $h\in H$ and $\varphi\in C_c(Y)$
    \begin{equation}
      \int_Y \varphi(h[y])d\tau_z(y)=\chi(h)^{-1} \int_Y
      \varphi(y)d\tau_z(y)\label{eq:14};
    \end{equation}
  \item for every $\varphi\in C_c(Y)$
    \begin{equation*}
      \int_Y
      \varphi(y)dy=\int_Z\Bigl(\int_Y
      \varphi(y)d\tau_z(y)\Bigr) d\lambda(z)
    \end{equation*}
  \end{enumerate}
\end{prop}
As in~\eqref{eq:6}, the following direct integral decompositions hold true
\begin{align}
L^2(Y,dy) & =\intl_Z L^2(Y,\tau_z) d\lambda(z) \nonumber \\
L^2(X,dx) & =\intl_Z \Bigl( \intl_Y L^2(X,\nu_y)d\tau_z(y)
\Bigr)d\lambda(z). \label{eq:8bis}
  \end{align}
If $z\notin\pi(Y)$, then $\tau_z=0$ and $L^2(Y,\tau_z)=\{0\}$. By \eqref{eq:8bis}, for any $f\in L^2(X,dx)$ we write 
\begin{align}
f &= \int_Z f_z\ d\lambda(z) \qquad f_z \in \int_Y L^2(X,\nu_y)d\tau_z(y)\label{slice1}\\
f_z &=  \int_Y f_{z,y} \, d\tau_z(y),\qquad f_{z,y}\in L^2(X,\nu_y).  \label{slice2}
\end{align}

In order to write the admissibility conditions, we introduce a unitary
operator $S$ that plays a crucial r\^ole because it implements the diagonalization of $U$.
Some extra ingredient is required.
Given $z\in \pi(Y)$, we take $o(z)$ as the origin of the orbit
$\pi^{-1}(z)=H[o(z)]$, which is
locally closed by assumption. Hence,  there exists a Borel
section  $q_z: \pi^{-1}(z)\to H$ such that  $q_z(o(z))$ is the identity element of $H$
and $q_z(y)[o(z)]=y$ for all $y\in
H[o(z)]$. Since $\tau_z$ is concentrated on $\pi^{-1}(z)$,  $q_z$ can
be extended as a Borel map on $Y$ such that the equality
$q_z(y)[o(z)]=y$ holds true for $\tau_z$-almost all
$y\in Y$.  Furthermore, by~\eqref{PHI},  the level set $\Phi^{-1}(o(z))$
is invariant under  the action of the
stability subgroup $H_z:=H_{o(z)}$  at $o(z)$, the measure
$\nu_z:=\nu_{o(z)}$  is concentrated on $\Phi^{-1}(o(z))$ and, by~\eqref{eq:13},
it is $H_z$-relatively invariant with weight $\chi\beta$.  Finally, we fix the
Haar measure $ds$ on $H_z$ in such a way that   Weil's formula holds true, namely,
\begin{equation}
\intl_H \varphi(h) \chi(h)^{-1} \, dh =\intl_Y \Bigl(\intl_{H_z} \varphi(q_z(y)\,s) ds\Bigr)d\tau_z(y),
\qquad\varphi\in C_c(H).\label{weil}
\end{equation}
We begin to build $S$  using the notation introduced in \eqref{slice1} and \eqref{slice2}. 
Lemma~5 of \cite{dede10} shows that for
$\lambda$-almost every $z\in Z$ the operator
\begin{align}
& S^z:\intl_Y L^2(X,\nu_y)d\tau_z(y)\to  L^2(Y\times X,\tau_z\otimes \nu_z)\nonumber\\
\label{eq:Sz2}
& (S^z f_z)(y,x)=\sqrt{\chi(q_z(y))\beta(q_z(y))}f_{z,y}(q_z(y).x),
\end{align}
is unitary;  since $\lambda$ is concentrated on
$\pi(Y)$, if $z\not\in\pi(Y)$ we  set $L^2(Y\times X,\tau_z\otimes \nu_z) =\set{0}$ and $S^z=0$.  
It follows that the direct integral operator
\begin{align*}
  & S:L^2(X,dx)\to  \int_{ Z}  L^2(Y\times X,\tau_z\otimes \nu_z) \,d\lambda(z) \\
  & S f =  \int_{ Z} S^{z} f_{z} \,d\lambda(z)
\end{align*}
is unitary as well.

\begin{remark}\label{rem:diagon}
As explained in \cite{dede10}, $S$  diagonalizes the representation $U$. More precisely,
for $z\in\pi(Y)$ define $\pi^z$ as the quasi-regular representation of $H_z$ acting on
$L^2(X,\nu_z) $, namely, for  $s\in H_z$ and $u\in L^2(X,\nu_z)$
\[
(\pi^z_su)(x)=\sqrt{\chi(s^{-1})\beta(s^{-1})} u(s^{-1}.x).
\]
Extend $\pi^z$ to a
representation of $\Sigma\rtimes H_z$ by mapping  $y\mapsto
e^{-2\pi i\scal{y}{o(z)}} \operatorname{Id}$, where
$\operatorname{Id}$ is the identity operator on $L^2(X,\nu_z) $.  Then,  the induced
representation $\operatorname{Ind}_{\Sigma\rtimes H_{z}}^G(e^{-2\pi i\scal{\cdot}{o(z)}}\otimes\pi^z)$
acts on $L^2(Y\times X,\tau_z\otimes
\nu_z)=L^2\bigl(Y,\tau_z;L^2(X,\nu_z)\bigr)$ and the operator $S$ intertwines $U$
with the direct integral representation of $G$ given by
$\int_{ Z}    \operatorname{Ind}_{\Sigma\rtimes H_{z}}^G(e^{-2\pi i\scal{\cdot}{o(z)}}\otimes\pi^z) \,d\lambda(z)$.
\end{remark}


We are now ready to state the  characterization of the
admissible vectors.
\begin{theorem}[\cite{dede10}, Theorem~6]\label{adm-vec1} A function $\eta\in L^2(X,dx)$ is an
  admissible vector for $G$ if and only if
for $\lambda$-almost every $z\in  Z$ and for every $u\in L^2(X,\nu_z)$
\begin{align}
\|u\|_{\nu_{z}}^2 & =
     \int_{Y}\Bigl(\int_{H_{z}}
  \bigl|{\int_X u(x)\overline{S^z\eta_z}(y,s^{-1}.x)d\nu_z(x)}\bigr|^2\chi(s^{-1})\beta(s^{-1})\,ds\Bigr)
\,\frac{\chi(q_z(y))}{\Delta_H(q_z(y))} d\tau_z(y). \label{SECOND}
\end{align}
\end{theorem}
According to~\eqref{eq:11}, the measure $\nu_z$ lives on 
$\Phi^{-1}(o(z))$ and its restriction to this manifold has density
$(J\Phi)^{-1}$ with respect to the corresponding Riemannian measure. 
Moreover, if $d=\dim{\Sigma}$ the Jacobian criterion implies
that $\Phi^{-1}(o(z))$ is a finite set, $\nu_z$ is a linear combination of
Dirac measures and $L^2(X,\nu_z)$ becomes a finite
dimensional vector space.  These facts provide a stronger version of
Theorem~\ref{adm-vec1}, which is very useful in the examples.
\begin{theorem}[\cite{dede10}, Corollary 4]\label{th:corollary 4}
Assume that $\dim{\Sigma}=d$. A function $\eta\in L^2(X,dx)$ is an admissible
vector for $G$ if and only if for $\lambda$-almost every $z\in Z$ and
 all points $x_{1},\ldots, x_{M}\in\Phi^{-1}(o(z))$ 
\begin{equation}
\int_H \eta(h^{-1}.x_{i})\overline{ \eta(h^{-1}.x_{j})}
\frac{dh}{\chi(h)\beta(h)}= (J\Phi)(x_{i}) \ \delta_{ij}\qquad
i,j=1,\ldots M.\label{eq:corollary 4}
\end{equation}
\end{theorem}

\subsection{Compact stabilizers}\label{sub:compact}
In this section, we make a compactness assumption that
allows us to refine~Theorem~\ref{adm-vec1}.
\begin{ass}\label{H2}
For almost every $y\in \Phi(\R^d)$ the stability subgroup
$H_y$ of
  $y$ is compact.
\end{ass}
This assumption allows to decompose $U$ into its irreducible
components by means of $S$. Indeed,  for $\lambda$-almost every $z\in
Z$,  $H_z$ is compact, so that the
representation $\pi^z$ of $H_z$  is completely reducible and, hence,
we can decompose it as direct sum of its irreducible components
\[
\pi^z\simeq \bigoplus_{i\in I} m_i\, \pi^{z,i}\qquad L^2(X,\nu_z)\simeq\bigoplus_{i\in I} m_i\, \C^{d_i^z},
\]
where each $\pi^{z,i}$ is an irreducible representation of
$H_z$
acting on some $\C^{d_i^z}$, any two of them are not equivalent,  and the cardinal $m_i \in\N\cup
\{\aleph_0\}$ is the multiplicity of $\pi^{z,i}$ into $\pi^z$.
Theorem~7 in \cite{dede10} shows that
it is possible to choose the index set $I$ in such a way that $m_i$ is
independent of $z$ (it can happen that $d_i^z=0$). 
Hence, for $\lambda$-almost all $z\in Z$,
\begin{equation}
 L^2(Y\times X,\tau_z\otimes \nu_z) \simeq\bigoplus_{i\in I} m_i
 L^2\bigl(Y,\tau_z;\C^{d_i^z}\bigr).\label{decom1} 
\end{equation}
Denote by $\{e_j\}_{j=1}^\infty$ the canonical basis of
$\ell_2=:\C^{\aleph_0}$ and, for each $i\in I$ regard $\C^{m_i}$ as  subspace of $\ell_2$
spanned by the family $\{e_j\}_{j=1}^{m_i}$. Then,  for any $f\in L^2(X,dx)$ we write
\begin{equation}\label{eq:decomp_Sf}
Sf= \sum_{i\in I} \sum_{j=1}^{m_i} \int_Z F_z^{ij}\,d\lambda(z)  \otimes e_j
\end{equation}
where  $F_z^{i,j}\in L^2(Y,\tau_z,\C^{d_i^z})$ for all $z\in Z$, $i\in
I$ and $j=1,\ldots,m_i$, and each of the sections  $ z\mapsto F^{ij}_z$ is
$\lambda$-measurable, {\em i.e.}  for all
$\varphi\in C_c(X)$ the map
 \[ z\mapsto \int_{Y\times X}
F^{ij}_z(y,x)  S^z\varphi(y,x)d\tau_z(y)d\nu_z(x)\]
is $\lambda$-measurable, where $F^{ij}_z$ is regarded as an element of
$L^2(Y\times X,\tau_z\otimes \nu_z) $ by~\eqref{decom1}  (see remark
before Theorem~5 and Remark~9 of \cite{dede10}). Note that, since $S$
is
unitary, then clearly
\[ \intl_X \abs{f(x)}^2dx = \sum_{i\in I }\sum_{j=1}^{m_i}\intl_Z\intl_Y \nor{F^{ij}_z(y)}^2_{\C^{d_i^z}} \,d\tau_z(y)d\lambda(z).\]

We are now ready to characterize the existence of admissible vectors
under Assumption~\ref{H2}.
\begin{theorem}[Theorem~9 \cite{dede10}]\label{th:9}
A function $\eta\in L^2(\R^d,dx)$ is an admissible vector if and only
if  for all $i\in I$ and $j=1,\ldots,m_i$ there exists  $\lambda$-measurable
sections
\[ Z\ni z\mapsto F^{ij}_z\in L^2(Y,\tau_z,\C^{d_i^z})\]
such that
\begin{align*}
\sum_{i\in I }\sum_{j=1}^{m_i}\intl_Z\intl_Y \nor{F^{ij}_z(y)}^2_{\C^{d_i^z}} \,d\tau_z(y)d\lambda(z)<+\infty
\end{align*}
and, for almost every $z\in Z$ and for all $i\in I$,  $j,\ell=1,\ldots,m_i$
\begin{align}\label{eq:th9}
\intl_Y \scal{F^{ij}_z(y)}{F^{i\ell}_z(y)}_{\C^{d_i^z}} \frac{\chi(q_z(y))}{\Delta_H(q_z(y))}\,d\tau_z(y) =\frac{d_i^z}{\operatorname{vol}{H_z}}\delta_{j\ell}.
\end{align}
Under the above circumstances
\[
S\eta = \sum_{i\in I} \sum_{j=1}^{m_i} \int_Z F_z^{ij}\,d\lambda(z) \otimes e_j.
\]
\end{theorem}
\section{Orbit equivalence}\label{orbitequivalence}
In this section we introduce the notion of orbit map and, more importantly, of orbit equivalence,
which is devised in order to treat in a unified ways families of groups depending on a parameter. 

Suppose we have two sets of data $(\Sigma\rtimes H, dh, X, Y)$ and $(\Sigma'\rtimes H', dh', X', Y')$, each enjoying all the properties that we have discussed in the previous section, where $dh$ and $dh'$ are two chosen left Haar measures. These data uniquely determine  $\chi$, $\beta$ and $\Phi$, and the analogues for the second set, which will all be denoted with primed letters. We assume now that the two sets are related by a triple of maps  $(\theta,\psi,\xi)$,  called an {\it orbit map} between the two sets, which are required to satisfy the following properties:
\begin{itemize}
\item[(OM1)] $\theta:H\to H'$ is an isomorphism between the two connected Lie groups $H$ and $H'$ such that  $dh'=\theta_* dh$ (the image measure); 
\item[(OM2)] $\psi:X\to X'$ is a diffeomorphism such that $\psi(hx)=\theta(h)\psi(x)$ for every $h\in H$ and every $x\in X$; hence in particular, by invariance of domain,  $d=d'$;
\item[(OM3)] $\xi:Y\to Y'$ is a diffeomorphism such that $\xi(\Phi(x))=\Phi'(\psi(x))$ for every $x\in X$; hence in particular, by invariance of domain,  $n=n'$.
\end{itemize}
\begin{definition}\label{OEQ} An orbit map $(\theta,\psi,\xi)$ between $(\Sigma\rtimes H, dh, X, Y)$ and $(\Sigma'\rtimes H', dh', X', Y')$ is called an {\it orbit equivalence} if the Jacobian $J\xi$
is a positive constant on $Y$. 
\end{definition}

The basic properties of orbit maps $(\theta,\psi,\xi)$ are recorded below.
\begin{prop}\label{OMprop} Let $(\theta,\psi,\xi)$ be an orbit map. Then
\begin{itemize}
\item[(i)] $\xi(h[y])=\theta(h)[\xi(y)]$ for every $h\in H$ and every $y\in Y$;
\item[(ii)]  $\psi(\Phi^{-1}(y))=(\Phi')^{-1}(\xi(y))$ for every $y\in Y$;
\item[(iii)] $J\xi(h[y])\chi(h)^{-1}=\chi'(\theta(h))^{-1}J\xi(y)$  for every $h\in H$ and every $y\in Y$;
\item[(iv)] $J\psi(hx)\beta(h)=\beta'(\theta(h))J\psi(x)$ for every $h\in H$ and every $x\in X$; 
\end{itemize}
\end{prop}
\begin{proof} (i) Using the various properties, including the surjectivity of $\Phi:X\to Y$, we infer
\begin{align*}
\xi(h[y])&=\xi(h[\Phi(x)])\\
&=\xi(\Phi(hx))\\
&=\Phi'(\psi(hx))     \\
&=\Phi'(\theta(h)\psi(x))\\
&=\theta(h)[\Phi'(\psi(x)]\\
&=\theta(h)[\xi(y)].
\end{align*}
\noindent
(ii) This follows immediately from (OM3), since
\[
\Phi'(\psi(\Phi^{-1}(y)))=\xi(\Phi(\Phi^{-1}(y)))=\xi(y).
\]
\noindent
(iv) Taking tangent maps at $x$, we have $\psi_{*hx}h=\theta(h)\psi_{*x}$. The absolute value of the determinant of  these yields the  equality asserted in (iv). Statement (iii) follows from (i).  
\end{proof}
\begin{remark}\label{orbitmapping} By (OM3),  the diagram
 $$
 \begin{CD} 
 X @>\Phi>> Y\\
 @VV\psi V@VV\xi V \\
X'@>\Phi'>> Y' \end{CD}
$$
is commutative.  It is also equivariant with respect to the actions of $H$ (on $X$ and $Y$) and those of $H'$ (on $X'$ and $Y'$). The equivariance in the left column is (OM2) and that of the right column is (i) of Proposition~\ref{OMprop}. 
 This means that the $H$-orbits in $X$ are mapped onto the $H'$-orbits in $X'$ by $\psi$,  and that the $H$-orbits in $Y$ are mapped onto the $H'$-orbits in $Y'$ by $\xi$, whence the name orbit map. Finally, by  (ii) of Proposition~\ref{OMprop}, fibers relative to $\Phi$ and $\Phi'$ are mapped onto eachother by $\psi$.
 \end{remark}
 
\begin{remark}\label{samealfa} If $(\theta,\psi,\xi)$ is an orbit
  equivalence, then (iii) of Proposition~\ref{OMprop} implies that, for
  all $h\in H$, $\chi(h)=\chi'(\theta(h))$. Thus, two  orbitally equivalent data sets must have the ``same'' $\chi$. \end{remark}

\begin{lemma}\label{measurenu} Suppose that  $(\theta,\psi,\xi)$ is an orbit map. Then, for every $y'\in Y'$, the measure $\nu'_{y'}$ described by Proposition~$\ref{prop:dis1}$
is given by
\begin{equation}
\nu'_{y'}=J\xi^{-1}(y')\,J\psi\circ\psi^{-1}\cdot \psi_{*}(\nu_{\xi^{-1}(y')}),
\label{nuprime}
\end{equation}
where $\psi_{*}(\nu_{\xi^{-1}(y')})$ is the image measure of $\nu_{\xi^{-1}(y')}$ under $\psi$.
\end{lemma}
\begin{proof} We observe {\it en passant} that $J\psi\circ\psi^{-1}$  is a smooth density and $J\xi^{-1}(y')$ is of course a real number. Theorem~2 in \cite{dede10} tells us that the family $\{\nu_{y'}\}$ is unique provided that it satisfies:
\begin{itemize}
\item[(i)] $\nu_{y'}$ is concentrated on $\Phi^{-1}(y')$ for all $y'\in Y'$;
\item[(ii)] the family $\{\nu_{y'}\}$ is scalarly integrable with respect to $dy'$ and $dx'=\int_{Y'}\nu_{y'}\,dy'$;
\item[(iii)] for any $\varphi\in C_c(X')$ the map $y'\mapsto\int_{X'}\varphi(x')d\nu_{y'}(x')$ is continuous.
\end{itemize}
We therefore show that the measure $\bar\nu_{y'}$ defined by the right-hand side of \eqref{nuprime} satisfies these properties.  First we prove that $\bar\nu_{y'}$ is concentrated on $(\Phi')^{-1}(y')$ by showing that such is $\psi_{*}(\nu_{\xi^{-1}(y')})$. Indeed, for any Borel set $E'\subset X'$, if $y=\xi^{-1}(y')$ by (ii) of Proposition~\ref{OMprop} we have
\[
\psi^{-1}(E')\cap\Phi^{-1}(y)=\emptyset
\iff
E'\cap\psi(\Phi^{-1}(y))=\emptyset
\iff
E'\cap(\Phi')^{-1}(y')=\emptyset.
\]
As for (iii), take $\varphi\in C_c(X')$. Then since $x\mapsto\bar\varphi(x)=J\psi(x)\,\varphi(\psi(x))$ is in $C_c(X)$,
by definition of image measure we have
\begin{align*}
  \int_{X'}\varphi(x')\,d\bar\nu_{y'}(x') & =J\xi^{-1}(y')\int_XJ\psi(x)\varphi(\psi(x))\,d\nu_{\xi^{-1}(y')}(x)
  \\
&
=J\xi^{-1}(y')\left[\int_X\bar\varphi(x)\,d\nu_{y}(x)\right]_{y=\xi^{-1}(y')}.
\end{align*}
The smoothness of $\xi$ and the continuity of the square bracket with respect to $y$, implied by property (iii) of the family $\{\nu_y\}$, shows the desired continuity. Finally, to prove (ii) we  take $\varphi\in C_c(X')$ and compute
\begin{align*}
\int_{Y'}\int_{X'}\varphi(x')\,d\bar\nu_{y'}(x')\,dy'
&=\int_{Y'}J\xi^{-1}(y')\left[\int_XJ\psi(x)\varphi(\psi(x))\,d\nu_{\xi^{-1}(y')}(x)\right] \,dy'    \\
(y=\xi^{-1}(y'))\hskip0.3truecm&=\int_{Y}\left[\int_XJ\psi(x)\varphi(\psi(x))\,d\nu_{y}(x)\right] \,dy     \\
\eqref{eq:12}\hskip0.3truecm&=\int_X J\psi(x)\varphi(\psi(x))\,dx   \\
(x'=\psi(x))\hskip0.3truecm&=\int_{X'}\varphi(x')\,dx',
\end{align*}
as desired.
\end{proof}
We introduce the following notation: $W:L^2(X',dx')\to L^2(X,dx)$ denotes the unitary map
\[
Wf'(x)=\sqrt{J\psi(x)}f'(\psi(x)).
\]
We are finally in a position to state our main result on orbit equivalence.
\begin{theorem}\label{OE} Let $(\theta,\psi,\xi)$ be an orbit map between $(\Sigma\rtimes H, dh, X, Y)$ and $(\Sigma'\rtimes H', dh', X', Y')$. If $\eta\in L^2(X,dx)$ is admissible for $G=\Sigma\rtimes H$ and if
\begin{equation}
\eta^*=\sqrt{J\xi\circ\Phi}\;\eta
\label{etastar}
\end{equation}
is also in $L^2(X,dx)$, then $W^{-1}\eta^*$ is admissible for $G'=\Sigma'\rtimes H'$. If  $(\theta,\psi,\xi)$ is an orbit equivalence, then all admissible vectors of $G'$ are of the form $W^{-1}\eta^*$ for some admissible $\eta\in L^2(X,dx)$.
\end{theorem}
\begin{proof} Take $f,\eta\in L^2(X,dx)$ and put, for $y\in Y$ and $h\in H$
\[
\omega_{f,\eta}(y,h)=\int_Xf(x)\overline{\eta(h^{-1}x)}\,d\nu_{y}(x).
\]
By Lemma~1 in \cite{dede10}, $\eta$ is admissible if and only if for every $f\in L^2(X,dx)$
\begin{equation*}
\int_Y\int_H|\omega_{f,\eta}(y,h)|^2\,\frac{dh\,dy}{\chi(h)\beta(h)}=\|f\|_2^2.
\end{equation*}
Evidently, an analogous statement holds for the primed set. Take next $f',\eta'\in L^2(X',dx')$. By 
Lemma~\ref{measurenu} and (iv) of Proposition~\ref{OMprop}
\begin{align*}
\omega_{f',\eta'}(y',\theta(h))&=\int_{X'}f'(x')\overline{\eta'(\theta(h)^{-1}x')}\,d\nu_{y'}(xÕ)\\
&=\int_{X'}f'(x')\overline{\eta'(\theta(h)^{-1}x')}\,J\xi^{-1}(y')\,J\psi(\psi^{-1}(x'))\; d\left[\psi_{*}\nu_{\xi^{-1}(y')}\right](x')\\
(x=\psi^{-1}(x'))\hskip0.3truecm
&=J\xi^{-1}(y')\int_{X}f'(\psi(x))\overline{\eta'(\theta(h)^{-1}\psi(x))}\,\,J\psi(x)\; d\nu_{\xi^{-1}(y')}(x)\\
\text{(OM2)}\hskip0.3truecm&=J\xi^{-1}(y')\int_{X}Vf'(x)\overline{V\eta'(h^{-1}x))}\sqrt{\frac{J\psi(x)}{J\psi(h^{-1}x)}}\, d\nu_{\xi^{-1}(y')}(x)\\
&=J\xi^{-1}(y')\sqrt{\frac{\beta'(\theta(h))}{\beta(h)}}\,\omega_{Vf',V\eta'}(\xi^{-1}(y'),h).
\end{align*}
Now the requirement on Haar measures in (OM1) comes finally  into play:
\begin{align*}
\int_{Y'}\int_{H'}\left|\omega_{f',\eta'}(y',h')\right|^2\,\frac{dh'\,dy'}{\chi'(h')\beta'(h')}
&= \int_{Y'}\int_H\left|\omega_{f',\eta'}(y',\theta(h))\right|^2\,\frac{dh\,dy'}{\chi'(\theta(h))\beta'(\theta(h))}\\
&= \int_{Y'}\int_H J\xi^{-1}(y')^2\left|\omega_{Vf',V\eta'}(\xi^{-1}(y'),h)\right|^2    \,\frac{dh\,dy'}{\chi'(\theta(h))\beta(h)}\\
\hskip-0.3truecm (\xi(y)=y')\hskip0.3truecm
&= \int_{Y}\int_H\left|\omega_{Vf',V\eta'}(y,h)\right|^2\,\frac{1}{J\xi(y)}\,\frac{dh\,dy}{\chi'(\theta(h))\beta(h)}\\
&= \int_{Y}\int_H \left|\omega_{Vf',V\eta'}(y,h)\right|^2\frac{1}{J\xi(h^{-1}[y])}\,\frac{dh\,dy}{\chi(h)\beta(h)},
\end{align*}
where in the last line we have used (iii) of Proposition~\ref{OMprop}, namely
\[
J\xi(y)\chi'(\theta(h))=J\xi(y)\chi'(\theta(h^{-1}))^{-1}=J\xi(h^{-1}[y])\chi(h).
\]
Observe now that since $\nu_y$ is concentrated on $\Phi^{-1}(x)$,
\begin{align*}
\frac{\omega_{f,\eta}(y,h)}{\sqrt{J\xi(h^{-1}[y])}}
&=\int_X f(x)\frac{\overline{\eta(h^{-1}x)}}{\sqrt{J\xi(h^{-1}[y])}}d\nu_y(x)\\
&=\int_X f(x)\frac{\overline{\eta(h^{-1}x)}}{\sqrt{J\xi(h^{-1}[\Phi(x)])}}d\nu_y(x)\\
&=\omega_{f,\eta^*}(y,h).
\end{align*}
We may finally conclude that
\begin{equation}
\int_{Y'}\int_{H'}\left|\omega_{f',\eta'}(y',h')\right|^2\,\frac{dh'\,dy'}{\chi'(h')\beta'(h')}
= \int_{Y}\int_H \left|\omega_{Vf',(V\eta')^*}(y,h)\right|^2\,\frac{dh\,dy}{\chi(h)\beta(h)}.
\label{orbitequality}
\end{equation}
Therefore, if $(W\eta')^*=\eta$ is admissible, then the right-hand side of \eqref{orbitequality}
 is equal to 
$\|Wf'\|_2^2=\|f'\|_2^2$. Hence, if 
$W\eta'=\sqrt{J\xi\circ\Phi}\,\eta \in L^2(X,dx)$, then $\eta'\in L^2(X',dx')$ because $W$ is unitary.
Under these circumstances, the left-hand side of \eqref{orbitequality} is thus equal to $\|f'\|_2^2$, 
and this says precisely that $\eta'$ is admissible. This proves the first statement of the theorem. 
The second one is clear because if $\sqrt{J\xi\circ\Phi}$ is constant, then 
$W\eta'$  is in $L^2(X,dx)$ if and only if $\eta \in L^2(X,dx)$, so $\eta'$ is admissible if and only if $\eta$ is such.
\end{proof}

\begin{remark}\label{iff} The very end of the proof of the theorem shows that under the assumption that $(\theta,\psi,\xi)$ is an orbit equivalence, that is, that  $J\xi$ is constant, the two groups $G$ and $G'$  are either both reproducing or neither of them is, and in the positive case the admissible vectors are in one-to-one correspondence. In the next section we will present many examples in which this situation occurs.\end{remark}

\begin{remark} If $\eta$ is an admissible vector for $G$ and has compact support, then \eqref{etastar} is trivially 
in $L^2$. Hence, if $G$ has an admissible vector with compact support, any group $G'$ belonging to a data set for which an orbit map between the data sets of $G$ and $G'$ exists is also reproducing.
\end{remark}
\section{Admissible vectors}\label{AV}
For the reader's convenience, we recall the classification of the groups in  $\mathcal
E_2$ obtained in \cite{albadede11} using a  more explicit notation. For $t\in\R$ we put
$$
R_t=\begin{bmatrix}
               \cos t & \sin t \\
               -\sin t & \cos t \\
             \end{bmatrix},\qquad  
   A_t=\begin{bmatrix}
               \cosh t & \sinh t \\
               \sinh t & \cosh t \\
             \end{bmatrix},
$$
and we define the subspaces of ${\rm Sym(2,\R)}$
\begin{align*}
{} & \Sigma_1=\bigl\{\begin{bmatrix}
u &0  \\
0&u \\
\end{bmatrix}:u\in\R\bigr\} &&
\Sigma_2=\bigl\{\begin{bmatrix}
u &0  \\
0&- u \\
\end{bmatrix}:u\in\R\bigr\} &&
\Sigma_3=\bigl\{
\begin{bmatrix}
                           u & 0\\
                            0 & 0 \\
                          \end{bmatrix}:
                          u\in\R\bigr\}  \\
& \Sigma_1^\perp=\bigl\{\begin{bmatrix}
u &v \\
v&-u \\
\end{bmatrix}:u,v\in\R\bigr\}
&& \Sigma_2^\perp=\bigl\{\begin{bmatrix}
u &v \\
v&u \\
\end{bmatrix}:u,v\in\R\bigr\}&&
\Sigma_3^\perp=\bigl\{
\begin{bmatrix}
                           0 & v\\
                            v & u \\
                          \end{bmatrix}:
                          u,v\in\R\bigr\}.
\end{align*}
The statement below is  a more readable version of Theorem~(1.1) in Part~I.
A comment on the notation. In four cases, we index a family of groups by means
of a parameter $\alpha\in[0,+\infty]$. The meaning for $\alpha=\infty$  is that of a limit as $\alpha\to+\infty$.
For example, in the case (2.1) we have  
$\Sigma_1\rtimes\{e^tR_{\alpha t}:t\in\R\}=
\Sigma_1\rtimes\{e^{s/\alpha}R_{s}:s\in\R\}$
so that the limit group is $\Sigma_1\rtimes\{R_{s}:s\in\R\}$.

\begin{theorem}[Theorem~1.1 \cite{albadede11}]\label{THELIST}
The following is a complete list, up to $Sp(2,\R)$-conjugation, of the
groups in $\mathcal E_2$ with $1\leq\dim\Sigma\leq2$.\\
Two dimensional groups:
\begin{enumerate}[(2.1)]
\item $ \Sigma_1\rtimes \{e^tR_{\alpha t}:t\in\R\}$, $\alpha\in[0,+\infty]$
\item $\Sigma_2\rtimes \{e^t A_{\alpha t}:t\in\R\}$, $\alpha\in[0,+\infty]$
\item $\Sigma_3\rtimes 
\bigl\{\left[\begin{smallmatrix}
1&0 \\
t&1 \\
\end{smallmatrix}\right]
:t\in\R\bigr\}$
\item  $\Sigma_3\rtimes 
\bigl\{e^t\left[\begin{smallmatrix}
                                         1 & 0 \\
                                         t & 1 \\
                                       \end{smallmatrix}\right]
:t\in\R\bigr\}$

\item $\Sigma_3\rtimes 
\bigl\{\left[\begin{smallmatrix}
                                         e^{\alpha t} & 0 \\
                                         0 & e^{(\alpha+1)t} \\
                                       \end{smallmatrix}\right]
:t\in\R\bigr\}$, $\alpha\in[-1,0]$
\end{enumerate}
Three dimensional groups:
\begin{enumerate}[(3.1)]
\item $\Sigma_1\rtimes \{e^tR_s :t,s\in\R\}$
\item $\Sigma_2\rtimes \{e^t A_s:t,s\in\R\}$
\item $\Sigma_3\rtimes 
\bigl\{\left[\begin{smallmatrix}
e^t & 0 \\
0 & e^s \\
\end{smallmatrix}\right]
:s,t\in\R\bigr\}$
\item $\Sigma_3\rtimes 
\bigl\{\left[\begin{smallmatrix}
e^t & 0 \\
s & e^t \\
\end{smallmatrix}\right]
:s,t\in\R\bigr\}$
\item $\Sigma_3\rtimes 
\bigl\{\left[\begin{smallmatrix}
                                         e^{\alpha t} & 0 \\
                                         s & e^{(\alpha+1)t} \\
                                       \end{smallmatrix}\right]
:s,t\in\R\bigr\}$, $\alpha\in\R$, $\alpha\neq1/2$
\item $\Sigma_1^\perp\rtimes \{e^tR_{\alpha t}:t\in\R\}$, $\alpha\in[0,+\infty]$
\item $\Sigma_2^\perp\rtimes \{e^tA_{\alpha t}:t\in\R\}$, $\alpha\in[0,+\infty]$
\item $\Sigma_3^\perp\rtimes \bigl\{\left[\begin{smallmatrix}
1&t \\
0&1 \\
\end{smallmatrix}\right]
:t\in\R\bigr\}$
\item $\Sigma_3^\perp
\rtimes \bigl\{e^t\left[\begin{smallmatrix}
                                         1 & t \\
                                         0 & 1 \\
                                       \end{smallmatrix}\right]
:t\in\R\bigr\}$
\end{enumerate}
Four dimensional groups:
\begin{enumerate}[(4.1)]
\item $\Sigma_3\rtimes 
\bigl\{\left[\begin{smallmatrix}
a&0 \\
b&c \\
\end{smallmatrix}\right]
:a,b,c\in\R,\;a>0,c>0\bigr\}$
\item $\Sigma_1^\perp\rtimes \{e^tR_s :t,s\in\R\}$
\item $\Sigma_2^\perp\rtimes  \{e^t A_s:t,s\in\R\}$
\item $\Sigma_3^\perp\rtimes\bigl\{\left[\begin{smallmatrix}
                                         e^{\alpha t} & s \\
                                         0& e^{(\alpha+1)t} \\
                                       \end{smallmatrix}\right]
:s,t\in\R\bigr\}$, $\alpha\in [-1,0]$
\end{enumerate}
Five dimensional groups:
\begin{enumerate}[(5.1)]
\item $\Sigma_3^\perp
\rtimes \bigl\{\left[\begin{smallmatrix}
c&b \\
0&a \\
\end{smallmatrix}\right]
:a,b,c\in\R,\;a>0,c>0\bigr\}$.
\end{enumerate}
\end{theorem}

We start by ruling out a number of groups in the preceding list, which will turn out to be complete. The fact that the other groups in the list are reproducing will be established in the next section, where the admissible vectors are computed. 

\begin{prop}\label{prop:not_repr}
The following groups are not reproducing
\begin{itemize}
\item $(2.1, \alpha=\infty)\quad \Sigma_1\rtimes \{R_t:t\in\R\}$
\item $(2.2, \alpha=\infty) \quad\Sigma_2 \rtimes \{A_t:t\in\R\}$
\item $(2.3)\quad \Sigma_3\rtimes \bigl\{\left[\begin{smallmatrix}
1&0 \\
t&1 \\
\end{smallmatrix}\right]
:t\in\R\bigr\}$
\item $(2.5, \alpha=0) \quad\Sigma_3\rtimes \bigl\{\left[\begin{smallmatrix}
                                         1 & 0 \\
                                         0 & e^{t} \\
                                       \end{smallmatrix}\right]
:t\in\R\bigr\}$

\item $(3.5,\alpha=1/2),\quad \Sigma_3\rtimes \bigl\{\left[\begin{smallmatrix}
                                         e^{ t} & 0 \\
                                         s & e^{3t} \\
                                       \end{smallmatrix}\right]
:s,t\in\R\bigr\}$
  \item $(3.6, \alpha=\infty) \quad\Sigma_1^\perp \rtimes \{R_t:t\in\R\}$

  \item $(3.7, \alpha=\infty)\quad \Sigma_2^\perp \rtimes \{A_t:t\in\R\}$

\item $(3.8)\quad \Sigma_3^\perp\rtimes \bigl\{\left[\begin{smallmatrix}
1&t \\
0&1 \\
\end{smallmatrix}\right]
:t\in\R\bigr\}$

\item $(4.4, \alpha=-1) \quad\Sigma_3^\perp\rtimes\bigl\{\left[\begin{smallmatrix}
                                         t & s \\
                                         0& 1 \\
                                       \end{smallmatrix}\right]
:s\in\R,t>0\bigr\}$
\item $(5.1)\quad \Sigma_3^\perp\rtimes \bigl\{\left[\begin{smallmatrix}
c&b \\
0&a \\
\end{smallmatrix}\right]
:a,b,c\in\R,\;a>0,c>0\bigr\}$
\end{itemize}
\end{prop}
\begin{proof}
All cases are treated using Theorem \ref{th:suffnec}.  We present all the details regarding the first group and we simply indicate what happens with the others.

Consider $\Sigma_1\rtimes \{R_t:t\in\R\}$. Assumption \ref{H1} is satisfied because the orbits are trivial. By \eqref{semidirect}, for $a,t\in\R$, we have $R_t[a\sigma_1]=a\sigma_1$ and so by \eqref{eq:alpha} we get $\alpha(R_t)=1$. Hence, since $\{R_t:t\in\R\}$ is Abelian, by \eqref{eq:deltaG}, $\Sigma_1\rtimes \{R_t:t\in\R\}$ is unimodular.  Therefore $\Sigma_1\rtimes \{R_t:t\in\R\}$ is not reproducing by Theorem \ref{th:suffnec}.

 Assumption \ref{H1} is always satisfied for all the remaining groups
 in the statement, and the case-by-case verification is
 straightforward. The following groups are not reproducing because
 they are unimodular.
 \begin{align*}
 & \Sigma_1^\perp \rtimes \{R_t:t\in\R\} 
&& \Sigma_2 \rtimes \{A_t:t\in\R\} 
&&   \Sigma_2^\perp \rtimes \{A_t:t\in\R\}  
\\ 
 & \Sigma_3\rtimes \bigl\{\left[\begin{smallmatrix}
1&0 \\
t&1 \\
\end{smallmatrix}\right]
:t\in\R\bigr\}  
&&  \Sigma_3\rtimes \bigl\{\left[\begin{smallmatrix}
                                         1 & 0 \\
                                         0 & e^{t} \\
                                       \end{smallmatrix}\right]
:t\in\R\bigr\} 
&&  \Sigma_3\rtimes \bigl\{\left[\begin{smallmatrix}
                                         e^{ t} & 0 \\
                                         s & e^{3t} \\
                                       \end{smallmatrix}\right]
\\
& \Sigma_3^\perp\rtimes \bigl\{\left[\begin{smallmatrix}
1&t \\
0&1 \\
\end{smallmatrix}\right]
:t\in\R\bigr\}  
&&  \Sigma_3^\perp\rtimes\bigl\{\left[\begin{smallmatrix}
                                         t & s \\
                                         0& 1 \\
                                       \end{smallmatrix}\right]
:s\in\R,t>0\bigr\} &&
 \end{align*}
The group $\Sigma_3^\perp\rtimes \bigl\{\left[\begin{smallmatrix}
c&b \\
0&a \\
\end{smallmatrix}\right]
:a,b,c\in\R,\;a>0,c>0\bigr\}$ is not reproducing because $n=d=2$ and for every $q\in\R$ the elements $\left[\begin{smallmatrix}
0&q \\
q&0 \\
\end{smallmatrix}\right]\in\Sigma_3^\perp$ have stabilizers that are not compact. These points form a (redundant) list of representatives for a set of orbits that fills up a set whose complement has measure zero.
\end{proof}
We are in a position to state our main result.
\begin{theorem}\label{main} The following is a complete list of non-conjugate\footnote{Within $Sp(2,\R)$.}
 reproducing groups in
${\mathcal E}_2$:
\begin{enumerate}[$(2D.1)$]
\item $\Sigma_1\rtimes \{e^tR_{\alpha t}:t\in\R\}$, $\alpha\in[0,+\infty)$
\item $\Sigma_2\rtimes \{e^t A_{\alpha t}:t\in\R\}$, $\alpha\in[0,+\infty)$
\item $\Sigma_3\rtimes 
\bigl\{\left[\begin{smallmatrix}
1&0 \\
t&1 \\
\end{smallmatrix}\right]
:t\in\R\bigr\}$
\item $\Sigma_3\rtimes \bigl\{\left[\begin{smallmatrix}
                                         e^{\alpha t} & 0 \\
                                         0 & e^{(\alpha+1)t} \\
                                       \end{smallmatrix}\right]
:t\in\R\bigr\}$, $\alpha\in[-1,0)$
\end{enumerate}
\begin{enumerate}[$(3D.1)$]
\item $\Sigma_1\rtimes \{e^tR_s :t,s\in\R\}$
\item $\Sigma_2\rtimes \{e^t A_s:t,s\in\R\}$
\item $\Sigma_3\rtimes \bigl\{\left[\begin{smallmatrix}
e^t & 0 \\
0 & e^s \\
\end{smallmatrix}\right]
:s,t\in\R\bigr\}$
\item $\Sigma_3\rtimes \bigl\{\left[\begin{smallmatrix}
e^t & 0 \\
s & e^t \\
\end{smallmatrix}\right]
:s,t\in\R\bigr\}$
\item $\Sigma_3\rtimes \bigl\{\left[\begin{smallmatrix}
                                         e^{\alpha t} & 0 \\
                                         s & e^{(\alpha+1)t} \\
                                       \end{smallmatrix}\right]
:s,t\in\R\bigr\}$, $\alpha\in\R\setminus\{1/2\}$
\item $\Sigma_1^\perp\rtimes \{e^tR_{\alpha t}:t\in\R\}$, $\alpha\in[0,+\infty)$
\item $\Sigma_2^\perp\rtimes \{e^tA_{\alpha t}:t\in\R\}$, $\alpha\in[0,+\infty)$
\item $\Sigma_3^\perp\rtimes \bigl\{e^t\left[\begin{smallmatrix}
                                         1 & t \\
                                         0 & 1 \\
                                       \end{smallmatrix}\right]
:t\in\R\bigr\}$
\end{enumerate}
\begin{enumerate}[$(4D.1)$]
\item $\Sigma_3\rtimes \bigl\{\left[\begin{smallmatrix}
a&0 \\
b&c \\
\end{smallmatrix}\right]
:a,b,c\in\R,\;a>0,c>0\bigr\}$
\item $\Sigma_1^\perp\rtimes \{e^tR_s :t,s\in\R\}$
\item $\Sigma_2^\perp\rtimes  \{e^t A_s:t,s\in\R\}$
\item $\Sigma_3^\perp\rtimes\bigl\{\left[\begin{smallmatrix}
                                         e^{\alpha t} & s \\
                                         0& e^{(\alpha+1)t} \\
                                       \end{smallmatrix}\right]
:s,t\in\R\bigr\}$, $\alpha\in (-1,0]$
\end{enumerate}
\end{theorem}

The proof proceeds by case by case analysis. In the course of the proof we also establish the full set of admissible vectors.
\subsection{Two dimensional groups}
 \paragraph{(2.1), $\alpha\in[0,+\infty)$} 
The elements of $\Sigma_1\rtimes H_\alpha$, with  
 $H_\alpha=\{h_t:=e^tR_{\alpha t}:t\in\R\}$, are 
\[
(u;t):=
\begin{bmatrix}e^{t}R_{\alpha t}&0\\ue^{t}R_{\alpha t}&e^{-t}R_{\alpha t}\end{bmatrix}.
\]
A left Haar measure of $H_\alpha$ is the pushforward of the Lebesgue measure $dt$ under $t\mapsto h_t$, and all $H_\alpha$ are unimodular. The action \eqref{semidirect} is
$h_t[y]=e^{2t}y$ and hence $\beta(t)=e^{2t}$ and $\chi(t)=e^{-2t}$. The intertwining is $\Phi(x)=-\|x\|^2/2$, with Jacobian is $J\Phi(x)=\|x\|$. Hence we put  $X_\alpha=\R^2\setminus\{(0,0)\}$ and $Y_\alpha=(-\infty,0)$.
The representation is
\[
U_{(u;t)}f(x_1,x_2)=e^{-t}e^{\pi
  iu(x_1^2+x_2^2)}f\left(e^{-t}R_{-\alpha t} (x_1,x_2)\right),
\]
here and in the following examples the vector $(x_1,x_2)$ has to be
understood as the column vector $[
\begin{smallmatrix}
  x_1 \\ x_2
\end{smallmatrix}]$.
These groups are mutually orbitally equivalent.   For any fixed $\alpha\not=0$, we define the orbit equivalence $(\theta,\psi,\xi)$ between $(\Sigma_1\rtimes H_0, dt, X_0, Y_0)$  and $(\Sigma_1\rtimes H_\alpha, dt, X_\alpha, Y_\alpha)$ as follows. First, $\theta:H_0\to H_\alpha$  is the group isomorphism $\theta(e^tI_2)=e^{t}R_{\alpha t}$. Further, we parametrize the elements in  $X_0$ and $X_\alpha$ with the  usual polar
 coordinates $(\rho,\theta)$ and define $\psi:X_0\to X_\alpha$ by $\psi(\rho,\theta)=(\rho,\theta+\alpha\log\rho)$.
 Finally, $\xi$ is the identity on $(-\infty,0)$. Both (OM1) and (OM3) are obvious, and for $h=e^tI_2\in H_0$
 \[
 \psi(h.(\rho,\theta))
 =\psi(e^t\rho,\theta)
 =(e^t\rho,\theta+\alpha(t+\log\rho))
=e^tR_{\alpha t}(\rho,\theta+\alpha\log\rho)
=\theta(h)[\psi(\rho,\theta)]
\]
establishes (OM2). 
Thus we may restrict ourselves to the case $\alpha=0$ and for simplicity we write $H$, $X$, $Y$ and so on, without the index $\alpha=0$.

The action \eqref{semidirect} is transitive on $Y$, whose origin is
chosen to be $o=-1/2$ and the corresponding stabilizer is
$H_o=\{I_2\}$, hence compact. Therefore the group is reproducing by Theorem~\ref{th:suffnec}. We shall use Theorem~\ref{th:9} in order to describe the admissible vectors. The
relevant fiber is $\Phi^{-1}(o)=S^1$, the unit circle, with measure $\nu_o=\,d\theta$, upon parametrizing its points by $(\cos\theta,\sin\theta)$.
We shall therefore identify $L^2(X,\nu_o)$ with $L^2(S^1,d\theta)$.
A smooth section $q:Y\to H$ for which $q(y)[o]=y$ is $q(y)=\sqrt{2|y|} I_2$. 
The measure $d\tau_o$ on $Y$ is $dy$. Because of Weil's integral
formula~\eqref{weil},  the  Haar measure of $H_o$  must be $\delta_{I_2}$. Next,
\[
S\eta (y,\theta)= \eta(\sqrt{2|y|}\cos\theta,\sqrt{2|y|}\sin\theta).
\]
Finally, the quasi regular representation of the singleton group $H_o$ is $\pi={\rm id}$ on $L^2(S^1,d\theta)$, hence it decomposes as the direct sum of countably many copy of the identity on $\C$.  Next, we  choose the exponential basis $\{(2\pi)^{-1/2}e^{ik\theta}:k\in\Z\}$ of 
$L^2(S^1,d\theta)$, so that
\[
F^k(y)=\frac{1}{\sqrt{2\pi}}\int_0^{2\pi}\eta(\sqrt{2|y|}\cos\theta,\sqrt{2|y|}\sin\theta)e^{-ik\theta}\,d\theta.
\]
 Theorem~\ref{th:9} tells us that $\eta\in L^2(X,dx)$ is admissible if and only if
\begin{equation*}
\sum_{k\in\Z}\int_{\R_-}|F^k(y)|^2dy<+\infty,
\qquad
\int_{\R_-}F^k(y)\overline{F^\ell}(y)\frac{dy}{|y|}=2\delta_{k\ell}.
\end{equation*}

\paragraph{(2.2), $\alpha\in[0,+\infty)$}\label{2.2}
The elements of $\Sigma_1\rtimes H_\alpha$, with  $H_\alpha=\{h_t=e^tA_{\alpha t}:t\in\R\}$, are 
\[
(u;t):=
\begin{bmatrix}e^{t}A_{\alpha t}&0\\
\left[\begin{smallmatrix}u&\\&-u\end{smallmatrix}\right]e^{t}A_{\alpha t}&e^{-t}A_{-\alpha t}\end{bmatrix}.
\]
A left Haar measure of $H_\alpha$ is the pushforward of the Lebesgue measure $dt$ under $t\mapsto h_t$ and $H_\alpha$ is unimodular. The action \eqref{semidirect} is
$h_t[y]=e^{2t}y$ and so $\beta(t)=e^{2t}$ and $\chi(t)=e^{-2t}$. 
 The intertwining is $\Phi(x)=-(x_1^2-x_2^2)/2$, with Jacobian $J\Phi(x)=\|x\|$. Thus, for every $\alpha$ we put  $X_\alpha=\R^2\setminus\{(x_1,x_2):x_1^2-x_2^2=0\}$ and $Y_\alpha=\R^*$.
The representation  is
\[
U_{(u;t)}f(x_1, x_2)
=e^{-t}e^{\pi iu (x_1^2- x_2^2)}f\bigl(e^{-t}A_{-\alpha t}(x_1,x_2)\bigr).
\]
The groups $\Sigma_1\rtimes H_\alpha$ are mutually orbitally equivalent. Fix $\alpha\not=0$. We define the orbit equivalence $(\theta,\psi,\xi)$ between $(\Sigma_1\rtimes H_0, dt, X_0, Y_0)$  and $(\Sigma_1\rtimes H_\alpha, dt, X_\alpha, Y_\alpha)$ as follows. First, $\theta:H_0\to H_\alpha$  is the group isomorphism $\theta(e^tI_2)=e^{t}A_{\alpha t}$. Now, the four connected components of $X_\alpha$ will be labeled by $(\pm,i)$, with $i=0,1$; the plus sign corresponds to the north-south quadrants and the minus sign to the east-west quadrants; $0$ labels north and east, and $1$ south and west. Each quadrant is then fibered by branches of hyperbolae. Thus, a point in a quadrant  is  parametrized by the branch of the hyperbola, labeled by $\rho>0$, and by the real variable $s$ running along it. Explicitly, write
\begin{equation}
\begin{cases}
x=\rho(\sinh s,(-1)^i\cosh s)&\text{for $x$ in the quadrant }(+,i)\\
x=\rho((-1)^i\cosh s,\sinh s)&\text{for $x$ in the quadrant }(-,i)
\end{cases}
\label{hypcoord}
\end{equation}
and denote by $(\rho,s)$ these hyperbolic coordinates on $X_\alpha$,
see Fig.~\ref{HC}.

\begin{figure}[htpb]
  \begin{center}
  \includegraphics{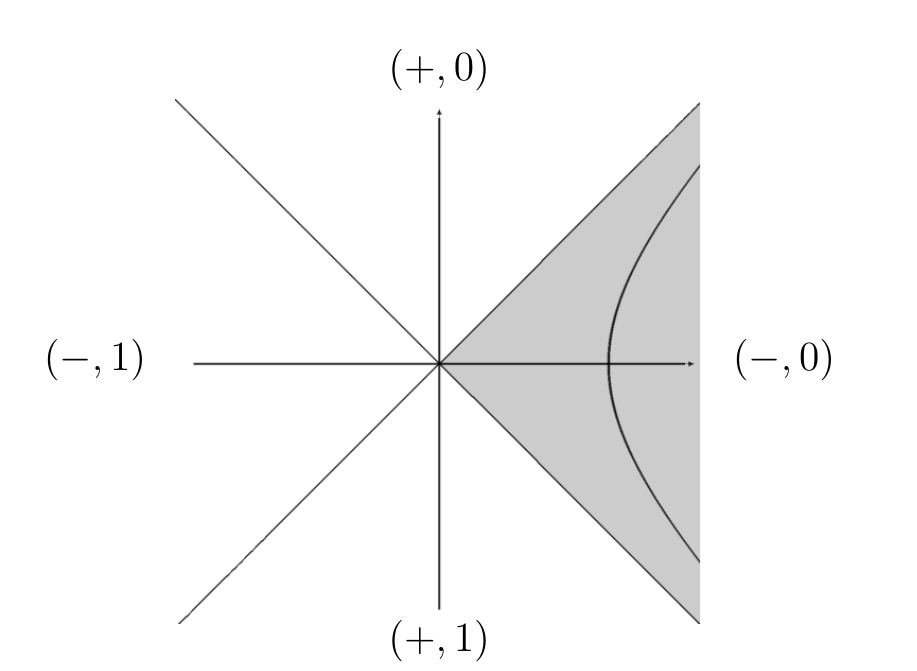} 
  \caption{Hyperbolic coordinates}\label{HC}
\end{center}
\end{figure}

The map $\xi$ is the identity of $\R^*$. The map $\psi:X_0\to X_\alpha$ is defined separately on each quadrant. For example, in the quadrant $(-,0)$ it is given by $\psi(\rho,s)=(\rho,s+\alpha\log\rho)$. Both (OM1) and (OM3) are obvious, and for $h=e^tI_2\in H_0$,
 \[
 \psi(h.(\rho,s))
 =\psi(e^t\rho,s)
 =(e^t\rho,s+\alpha(t+\log\rho))
=e^tA_{\alpha t}(\rho,s+\alpha\log\rho)
=\theta(h)[\psi(\rho,s)]
\]
establishes (OM2); the other three quadrants are treated
similarly.Thus, we may restrict ourselves to the case $\alpha=0$ and
for simplicty we write $H$, $X$, $Y$ and so on. 
The action \eqref{semidirect} on $Y$ has the two orbits 
$\R_-=(-\infty, 0)$  and $\R_+=(0, +\infty)$.
As origins we chose
$\pm1/2$ with stabilizers
$H_{\pm}=\{I_2\}$, hence compact.
Therefore the group is reproducing by Theorem~\ref{th:suffnec}.  We shall use Theorem~\ref{th:9} in order to describe the admissible vectors. The
relevant fibers are $\Phi^{-1}(1/2)$,
 the equilateral hyperbola with vertices in $(0, \pm1)$,
and $\Phi^{-1}(-1/2)$,
the equilateral hyperbola with vertices in $(\pm1, 0)$.
The measures $\nu_{\pm }$ on $X$ are concentrated each on the relative hyperbola and, because of  \eqref{eq:11}  are given, for functions in $C_c(\R^2)$, by
\begin{align}
\int_X\varphi(x_1,x_2)\,d\nu_{+}(x_1,x_2)
&=\sum_{i=0,1}\int_{\R}\varphi(\sinh s,(-1)^i\cosh s)\;\,ds \label{hyp1}    \\
\int_X\varphi(x_1,x_2)\,d\nu_{-}(x_1,x_2)
&=\sum_{i=0,1}\int_{\R}\varphi((-1)^i\cosh s,\sinh s)\;\,ds,\label{hyp2}
\end{align}
with the parametrization \eqref{hypcoord}.
We shall therefore identify each of $L^2(X,\nu_{+})$ and $L^2(X,\nu_{-})$ with  two copies 
of $L^2( \R,\,ds)$. A smooth section $q:Y\to H$ for which $q(y)[\pm 1/2]=y$ is $q(y)=\sqrt{2|y|} I_2$. Weil's integral formula~\eqref{weil} forces the  Haar measure of $H_{\pm}$  to be $\delta$. The measure $d\tau_\pm$ is $dy$ on $\R_\pm$. Next 
\[
S^{\pm}\eta_{\pm} (y,x)= \eta_{\pm}(\sqrt{2|y|}x),
\qquad
 y\in Y, x\in\Phi^{-1}(\pm1/2).
\]
Hereafter we shall write any function on $X$ as a sum of four functions, each supported on one of the connected components of $X$. The quasi regular representation $\pi^{\pm}$ of each singleton group $H_{\pm}$
is the identity on two copies of $L^2( \R,\,ds)$. Choose bases $\{e^{j}_{\pm,i}\}_{j\in\Z}$ of $L^2( \R,\,ds)$, one for each of the four basic branches and put
$F_{\pm,i}^j(y)=\langle S^{\pm}\eta_{\pm}(y,\cdot),e^{j}_{\pm,i}
\rangle$.
Theorem~\ref{th:9} tells us that $\eta\in L^2(X,dx)$ is admissible if and only if 
\begin{align*}
&\sum_{j\in\Z}\int_{\R_\pm}|F_{\pm,i}^j(y)|^2dy<+\infty \qquad i=0,1
&\int_{\R_\pm}F_{\pm,i}^j(y)\overline{F_{\pm,k}^\ell}(y)\frac{dy}{|y|}=2\delta_{j\ell}\delta_{ik}
\qquad i,k=0,1,\;j,\ell\in\Z.
\end{align*}

\paragraph{ (2.4)}
The elements of $G=\Sigma_3\rtimes H$, with $H=\set{h_t:=e^t[
  \begin{smallmatrix}
    1 & 0 \\ t & 0 
  \end{smallmatrix}]\,:\, t\in\R
}$
are 
\[
(u;t):=
\begin{bmatrix}e^{t}&0&0&0\\t e^t&e^t&0&0\\ue^{t}&0&e^{-t}&-t e^{-t}\\0&0&0&e^{-t}\end{bmatrix}.
\]
A left Haar measure of $H$ is the pushforward of the Lebesgue measure
$dt$ under $t\mapsto h_t$ and $H$ is unimodular.
The action~\eqref{semidirect} is $h_t[y]=e^{2t}y$, hence $\beta(t)=e^{2t}$ and $\chi(t)=e^{-2t}$.
 The intertwining is $\Phi(x)=-x_1^2/2$, with Jacobian $J\Phi(x)=|x_1|$. We put
$X=\R^2\setminus\{x_1=0\}$ and $Y=(- \infty, 0)$. The representation is
\[
U_{(u;t)}f(x_1, x_2)=e^{-t}e^{\pi iu x_1^2}f\left(e^{-t}(x_1, x_2- t x_1)\right).
\]
The action \eqref{semidirect} is transitive on $Y$, whose origin is chosen to be $o=-1/2$ with stabilizer
$H_{o}=\{I_2\}$, hence compact.
Therefore the group is reproducing by Theorem~\ref{th:suffnec}. We shall use Theorem~\ref{th:9} in order to describe the admissible vectors. The fiber $\Phi^{-1}(-1/2)$ consists of two vertical lines $\{(\pm1,x_2):x_2\in\R\}$.
The measure $\nu_o$ is the push-forward of one-dimensional Lebesgue measure to the straight line $\{(\pm1,x_2)\}$, and is given by the integral formula
\begin{equation}
\int_X\varphi(x_1,x_2)\,d\nu_{o}(x_1,x_2)
=\int_{\R}\varphi(1,s) ds
+\int_{\R}\varphi(-1,s)ds,
\qquad
\varphi\in C_c(\R^2),
\label{twolines}
\end{equation}
whereby we parametrize the lines of the
fiber by   $(\pm 1, s)$, with $s\in\R$.
A smooth section $q:Y\to H$ for which $q(y)[- 1/2]=y$ is 
\[
q(y)=\sqrt{2|y|}
\begin{bmatrix} 1 & 0\\
\log(\sqrt{2|y|}) & 1 \\
\end{bmatrix}.
\]
The measure $d\tau_o$ on $Y$ is $dy$. Weil's integral
formula~\eqref{weil}  forces the  Haar measure of $H_{o}$  to be the Dirac mass $\delta_{I_2}$. Next,
\[
S_\pm\eta(y, x_2)= \eta\left(\sqrt{2|y|}(\pm1, x_2\pm \log(\sqrt{2 |y|}))\right).
\]
Hereafter we shall write any function on $X$ as a sum of two functions, each supported
on one of the connected components of $X$, labeled by
``$+$'' for the right half plane and by ``$-$'' for the left half plane.
A point in a half plane, in turn, is  parametrized by the vertical line to which it belongs, determined by $y$,
and by the real variable $s$ running along it, as explained above.
Finally, the quasi regular representation of the singleton group $H_{o}$
is $\pi={\rm id}$ on each of the two copies of $L^2(\R, ds)$.
We choose two  bases $\{e_{\pm}^j\}_{j\in\Z}$ of
$L^2(\R , ds)$, one for each of the two basic lines, and put 
$F^j_\pm(y)=\scal{S\eta_\pm(y,\cdot)}{e_{\pm}^j}$.
Theorem~\ref{th:9} tells us that $\eta\in L^2(X,dx)$ is admissible if and only if
\begin{align*}
&\sum_{j\in\Z} \int_{\R_-}|F^j_\pm(y)|^2dy<+\infty
& \int_{\R_-}F_{\eps}^j(y)\overline{F_{\iota}^\ell}(y)\frac{dy}{|y|}=2\delta_{j\ell}\delta_{\eps\iota}
\qquad\eps,\iota\in\{\pm\},\;j,\ell\in\Z.
\end{align*}
\paragraph{ (2.5), $\alpha\in[-1,0)$} The elements of $G_\alpha:=\Sigma_3\rtimes H_\alpha$, with  
$H_\alpha=\{h_t:={\rm diag}(e^{\alpha t},e^{(\alpha+1)t}):t\in\R\}$, are of the form
\begin{equation*}
(u;t):=\begin{bmatrix}e^{\alpha t}&0&0&0\\0&e^{(\alpha +1)t}&0&0\\ue^{\alpha t}&0&e^{-\alpha t}&0\\0&0&0&e^{-(\alpha +1)t}\end{bmatrix}.
\end{equation*}
A left Haar measure  is the pushforward of the Lebesgue measure $dt$ under $t\mapsto h_t$ and $H$ is Abelian, hence unimodular. The action~\eqref{semidirect} is $h_t[y]=e^{2\alpha t}y$ and hence $\beta(t)=e^{(2\alpha+1)t}$ and $\chi(t)=e^{-2\alpha t}$. The intertwining is $\Phi(x)=-x_1^2/2$, with Jacobian $J\Phi(x)=|x_1|$.
Hence we put $X_{\alpha}=\R^2\setminus\{x_1=0\}$ and $Y_{\alpha}=(- \infty, 0)$. The representation is
\[
U_{(u;t)}f(x_1, x_2)=e^{-(2\alpha +1)t/2}e^{\pi iu x_1^2}f\left(e^{-\alpha t}x_1,
e^{-(\alpha +1)t}x_2\right).
\]
The groups are mutually orbitally equivalent.
For  $\alpha\in(-1,0)$, we define the orbit equivalence $(\theta,\psi,\xi)$ between $(G_{-1}, dt, X_{-1}, Y_{-1})$  and $(G_\alpha, dt, X_\alpha, Y_\alpha)$ as follows.
First, $\theta:H_{-1}\to H_\alpha$  is the group isomorphism 
$\theta({\rm diag}(e^{-t},1))={\rm diag}(e^{-t},e^{-(\frac{\alpha+1}{\alpha})t})$.
Next, $\psi:X_{-1}\to X_{\alpha}$ is given by 
$\psi(x_1,x_2)=(x_1,|x_1|^{\frac{\alpha+1}{\alpha}}x_2)$, whereas
$\xi:Y_{-1}\to Y_{\alpha}$ is $\xi(y)=y$. The verification of (OM1), (OM2) and (OM3) is straightforward,
and evidently $J\xi\equiv1$.
Thus, we may restrict ourselves to the case $\alpha=-1$ and we write $H$, $X$, $Y$ and so on.\\
The action \eqref{semidirect} is transitive on $Y$, whose origin is chosen to be $o=-1/2$ with stabilizer
$H_{o}=\{I_2\}$, hence compact. Therefore the group is reproducing by Theorem~\ref{th:suffnec}. 
We shall use Theorem~\ref{th:9} in order to describe the admissible vectors. The
 fiber $\Phi^{-1}(-1/2)$ consists of two vertical lines $\{(\pm1,x_2):x_2\in\R\}$.
The measure $\nu_o$ is again given by the integral formula \eqref{twolines}, 
whereby we parametrize the lines of the
fiber by   $(\pm 1, s)$, with $s\in\R$. A smooth section $q:Y\to H$ for which $q(y)[-1/2]=y$ is 
\[
q(y)=
\begin{bmatrix}\sqrt{2|y|}& 0\\0& 1\\\end{bmatrix}.
\]
Weil's integral
formula~\eqref{weil} forces the  Haar measure of $H_{o}$  to be $ \delta_{I_2}$. Next,
\[
S_\pm\eta (y, x_2)= (2|y|)^{-1/4}
\eta\left(\pm\sqrt{2|y|},x_2\right).
\]
Hereafter we shall write any function on $X$ as a sum of two functions, each supported
on one of the connected components of $X$, labeled by
``$+$'' for the right half plane and by ``$-$'' for the left half plane.
A point in a half plane, in turn, is  parametrized by the vertical line to which it belongs, determined by $y$,
and by the real variable $s$ running along it.
Finally, the quasi regular representation of the singleton group $H_{o}$
is $\pi={\rm id}$ on each of the two copies of $L^2(\R, ds)$.
We choose two  bases $\{e_{\pm}^j\}_{j\in\Z}$ of
$L^2(\R , ds)$, one for each of the two basic lines, and put 
$F^j_\pm(y)=\scal{S\eta_\pm(y,\cdot)}{e_{\pm}^j}$.
Theorem~\ref{th:9} tells us that $\eta\in L^2(X,dx)$ is admissible if and only if
\begin{align*}
&\sum_{j\in\Z}\int_{\R_-}|F_\pm^j(y)|^2dy<+\infty \\
&\int_{\R_-}F_{\eps}^j(y)\overline{F_{\iota}^\ell}(y)\frac{dy}{|y|}=\delta_{j\ell}\delta_{\eps\iota}
\qquad \eps,\iota\in\{\pm\},\;j,\ell\in\Z. 
\end{align*}

\subsection*{Three dimensional groups, $n=1$} 
\paragraph{(3.1)}
The elements of $\Sigma_1\rtimes H$, with $H=\{h_{t,\theta}:=e^tR_\theta:t\in\R, \theta\in[0,2\pi)\}$ are
\begin{equation*}
(u;t,\theta):=\begin{bmatrix}e^tR_{\theta}&0\\ue^tR_{\theta}&e^{-t}R_{\theta}\end{bmatrix}.
\end{equation*}
A left Haar measure of  $H$  is the pushforward under the map
$\R\times[0,2\pi)\to H$ defined by $(t,\theta)\mapsto h_{t,\theta}$ 
of the product Lebesgue measure $dt\,d \theta$, and $H$ is unimodular. 
The action  \eqref{semidirect} is $h_{t, \theta}[y]=e^{2t}y$, so that 
$\beta(t,\theta)=\beta(t)=e^{2t}$ and $\chi(t, \theta)=\chi(t)=e^{-2t}$. The intertwining is 
$\Phi(x)=-\|x\|^2/2$, with Jacobian  $J \Phi(x)=\|x\|$.
Hence we put $X=\mathbb{R}^2\setminus \{x=0\}$ and $Y=(-\infty,0)$.
The representation is
\[
U_{(u;t,\theta)}f(x_1,x_2)=e^{-t}e^{\pi iu(x_1^2+x_2^2)}f\left(e^{-t}R_{-\theta}(x_1,x_2)\right).
\]
The action \eqref{semidirect} is transitive on $Y$, whose origin is chosen to be $o=-1/2$, with stabilizer 
$H_o=\{(0, \theta), \, : \theta \in [0,2\pi)\}$, hence compact.
Therefore the group is reproducing by Theorem~\ref{th:suffnec}. We shall use Theorem~\ref{th:9} 
describe the admissible vectors. The fiber is $\Phi^{-1}(o)=\{x \in\R^2 \, : \|x\|=1\}=S^1$ with measure
$\nu_o=d \theta$ concentrated on the unit circle. This amounts to the  integral formula
\[
\int_X \varphi(x) d\nu_o(x) =\int_{0}^{2\pi} \varphi(\cos \theta, \sin \theta)\,d\theta,
\qquad
\varphi\in C_c(\R^2).
\]
We shall therefore identify $L^2(X,\nu_o)$ with $L^2(S^1,d \theta)$.
A smooth section $q:Y\to H$ for which $q(y)[o]=y$ is  $q(y)=\sqrt{2|y|}I_2$. The measure $d\tau$ is the Lebesgue measure $dy$. Because of Weil's integral
formula~\eqref{weil}, the  Haar measure of $H_o$  is $d\theta$, whence 
${\rm vol}(H_o)=2\pi$. Next,
\[
S\eta(y,(\cos\theta,\sin\theta))=\eta(\sqrt{2|y|}\cos \theta, \sqrt{2|y|}\sin \theta).
\]
The quasi regular representation $\pi$ of the group $H_o\simeq \mathbb T$
on $L^2(X, d \nu_o) \simeq L^2(S^1, d\theta)$
is completely reducible as the direct sum $\pi =\oplus_{\mathbb Z} \pi_k$
on the invariant subspaces ${\rm sp}\{e_k\} $,
where $\{e_k\}_{k\in\mathbb Z}$ is the normalized exponential
basis of $L^2(S^1, d\theta)$, namely $e_k(\theta)=e^{ik\theta}/\sqrt{2\pi}$.
Let $F^k(y)=\langle S\eta(y,\cdot),e_k\rangle$.
Theorem~\ref{th:9} tells us that $\eta\in L^2(X,dx)$ is admissible if and only if
\[
\int_{\R_-}\sum_{k\in\Z}|F^k(y)|^2\,dy<+\infty,
\qquad
\int_{\R_-}|F^k(y)|^2\frac{dy}{|y|}=\frac{1}{\pi}
\]
We can reformulate these by the natural change of variable
$r=\sqrt{2|y|}$.  Set
$\eta_k(r)=F^k(y)$, then 

\begin{align}
&\int_{\R_+}\sum_{k\in\Z}|\eta_k(r)|^2\,r\,dr<+\infty \qquad \label{3.1a}\\
&\int_{\R_+}|\eta_k(r)|^2\frac{dr}{r}=\frac{1}{2\pi}.\label{3.1b}
\end{align}
Notice that, under the chosen normalization,
\[
\eta_k(r)=\frac{1}{\sqrt{2\pi}}\int_0^{2\pi}\eta(r\cos\theta,r\sin\theta)\,e^{-ik\theta}\,d\theta,
\]
the ordinary Fourier coefficient of the restriction of $\eta$ to the circle of radius $r$.


\paragraph{(3.2)}
The elements of $\Sigma_2\rtimes H$, with $H=\{h_{t,s}:t,s\in\R\}$, are
\begin{equation*}
(u;t,s):=
\begin{bmatrix}e^t A_s&0\\
\left[\begin{smallmatrix}u&\\&-u\end{smallmatrix}\right]e^tA_s&e^{-t}A_{-s}
\end{bmatrix}.
\end{equation*}
A Haar measure of $H$ is the pushforward of the map $\R^2\to H$ under the map $(t,s)\mapsto h_{t,s}$ of the product Lebesgue measure $dt\,ds$. 
The action \eqref{semidirect} is $h_{t,s}[y]=e^{2t}y$,  $\beta(t,s)=\beta(t)=e^{2t}$ and 
$\chi(t,s)=\chi(t)=e^{-2t}$. The intertwining is $\Phi(x)=(x^2_2 - x^2_1)/2$, with Jacobian $J \Phi(x)=\|x\|$.
We put $X=\mathbb{R}^2\setminus \{(x_1, x_2) \in \R : \, x_1=\pm x_2\}$
and $Y=\mathbb R^*$. The representation is
\[
U_{(u;t,s)}f(x_1, x_2)=e^{-t}e^{\pi iu(x_1^2 -x_2^2)}f\left(e^{-t}A_{-s}(x_1, x_2)\right).
\]
The action \eqref{semidirect} on $Y$ has the two orbits  
$\R_-=(-\infty, 0)$  and $\R_+=(0, +\infty)$, whose origins are chosen to be 
$\pm1/2$ with stabilizers
 $H_{\pm}=\{(0, s), \, :s  \in \R \}$, hence not compact.
Therefore we shall use Theorem~\ref{adm-vec1} in order to show that $G$ is reproducing,
and then we describe the admissible vectors. The relevant fibers are: $\Phi^{-1}(1/2)$,
 the equilateral hyperbola with vertices in $(0, \pm1)$,
and $\Phi^{-1}(-1/2)$,
the equilateral hyperbola with vertices in $(\pm1, 0)$.
The measures $\nu_{\pm}$ on $X$ are concentrated each on the relative hyperbola and are given in integral form by \eqref{hyp1} and \eqref{hyp2}. As in Section~\ref{2.2}, 
we will identify each of $L^2(X,\nu_+)$ and $L^2(X,\nu_-)$ with  two copies
of $L^2( \R,ds)$.
A smooth section $q_{\pm}:\R_{\pm}\to H$ for which $q_{\pm}(y)[\pm1/2]=y$
is $q_{\pm}(y)=\sqrt{2|y|} I_2$. The measure $d\tau$ is the Lebesgue measure $dy$.
Weil's formula~\eqref{weil} forces the  Haar measure of $H_{\pm}$  to be $ds$.
Next,
\[
S^{\pm}\eta_{\pm} (y,x)= \eta_{\pm}(\sqrt{2|y|}x).
\]
Hereafter we shall write any function on $X$ as a sum of four functions, as in Section~\ref{2.2}. We will also adopt the following notation:
\[
\eta_{\pm, i}(y)(s) 
=  \begin{cases}
\eta_{+}\left( q_+(y)(\sinh s, (-1)^i\cosh s)\right) =
\eta_{+}\left( \sqrt{2|y|}(\sinh s, (-1)^i\cosh s)\right)& \\
\\
\eta_{-}\left( q_-(y)( (-1)^i\cosh s, \sinh s)\right)=
\eta_{-}\left(\sqrt{2|y|}((-1)^i\cosh s, \sinh s)\right).
\end{cases}
\]
Similarly, we shall write $u_{\pm,i}(s)$ to mean the analogous restrictions of $u\in L^2(X,\nu_\pm)$ to the various hyperbolae.
We may now apply Theorem~\ref{adm-vec1}.  Equation \eqref{SECOND} reads 
\begin{align}
\|u\|_{\pm}^2&=
\int_{\R_\pm} \int_{\R}
\Big| \int_{\R} \sum_{i=0,1}u_{\pm,i}(s) \overline{\eta_{\pm,i}(y)(s-(-1)^ip)}ds
\Big|^2 dp\frac{dy }{2|y|}\nonumber\\
&=\int_{\R_\pm} \int_{\R}
\Big| \sum_{i=0,1}\langle u_{\pm,i},\tau_{(-1)^ip}{\eta_{\pm,i}(y)}\rangle
\Big|^2 dp\;\frac{dy }{2|y|},
\label{adm32}
\end{align}
where $\tau_p$ denotes the standard translation on functions.
Choose first a function $u$ with support in the northern quadrant of $X$, so that $u=u_{+,0}$ and there is only one summand.
By setting $G_y(p)=\overline{\eta_{+,0}(y)}(-p)$, we get
\begin{align*}
\|u\|_+&=\int_{\R_+} \Bigl\{\int_{\R}
\Big|(u_{+,0}*G_y)(p)\Bigr|^2 dp\Bigr\}\,\frac{dy }{2y}     \\
&=\int_{\R_+}\|u_{+,0}*G_y\|_2^2\,\frac{dy }{2y}      \\
&=\int_{\R_+}\|\widehat{u_{+,0}}\;\widehat{G_y}\|_2^2\,\frac{dy }{2y}      \\
&=\int_{\R_+}\int_\R|\widehat{u_{+,0}}(\xi)|^2\;|\widehat{G_y}(\xi)|^2\,d\xi\,\frac{dy }{2y} \\
&=\int_{\R}|\widehat{u_{+,0}}(\xi)|^2\Bigl(\int_{\R_+}\;|\widehat{G_y}(\xi)|^2\,\frac{dy }{2y}\Bigr)  \,d\xi. 
\end{align*}
This is readily equivalent to requiring the inner integral to be equal to $1$ for almost every $\xi$. A similar argument for the other quadrants yields
\begin{equation}
\int_{\R_\pm}\frac{|\widehat{\eta_{\pm,i}(y)}(\xi)|^2}{|y|}\,dy=2,
\qquad \text{a.e. }\xi\in\R,\;i=0,1.
\label{3.2a}
\end{equation}
Next we use these in \eqref{adm32} and we obtain that the mixed terms must vanish, namely, using
Parseval's equality in the form
$\langle f,\tau_p g\rangle={\mathcal F}^{-1}(\hat f\,\overline{\hat g})(p)$
\begin{align*}
0&=\int_{\R_{\pm}}\int_{\R}
\langle u_{\pm,0},\tau_p{\eta_{\pm,0}(y)}\rangle
\overline{\langle u_{\pm,1},\tau_{-p}{\eta_{\pm,1}(y)}\rangle}\,dp\,\frac{dy }{2|y|}       \\
&=\int_{\R_{\pm}}\int_{\R} 
{\mathcal F}^{-1}\left(\widehat{u_{\pm,0}}\;\overline{\widehat{\eta_{\pm,0}(y)}}\right)(p)
\;
\overline{{\mathcal F}^{-1}\left(\widehat{u_{\pm,1}}\;\overline{\widehat{\eta_{\pm,1}(y)}}\right)(-p)}
\,dp\,\frac{dy }{2|y|}    \\
&=\int_{\R_{\pm}}\int_{\R} 
\widehat{u_{\pm,0}}(\xi)\;\overline{\widehat{\eta_{\pm,0}(y)}(\xi)}
\;
\overline{\widehat{u_{\pm,1}}(-\xi)}\;\widehat{\eta_{\pm,1}(y)}(-\xi)\,d\xi\,\frac{dy }{2|y|}    \\
&=\int_{\R}
\widehat{u_{\pm,0}}(\xi)\overline{\widehat{u_{\pm,1}}(-\xi)}
\Bigl( \int_{\R_{\pm}}
\widehat{\eta_{\pm,1}(y)}(-\xi)\;\overline{\widehat{\eta_{\pm,0}(y)}(\xi)}\,\frac{dy }{2|y|}\Bigr)
\,d\xi,   \\
\end{align*}
which yields the second set of admissibility conditions
\begin{equation}
\int_{\R_{\pm}}\frac{\widehat{\eta_{\pm,1}(y)}(-\xi)\;\overline{\widehat{\eta_{\pm,0}(y)}(\xi)}} {|y|}\,dy=0,
\qquad \text{a.e. }\xi\in\R,\;i=0,1.
\label{3.2b}
\end{equation}
It is not difficult to exhibit vectors $\eta$ that do satisfy \eqref{3.2a} and \eqref{3.2b}, thereby showing that the group is indeed reproducing.

\paragraph{(3.3)}  The elements of the group $\Sigma_3\rtimes H$, with $H=\{h_{a,c}:={\rm diag}(a^{-1/2},c):a,c>0\}$ are
\begin{equation*}
(u;a,c):=
\begin{bmatrix}a^{-1/2}&0&0&0\\0&c&0&0\\ua^{-1/2}&0&a^{1/2}&0\\0&0&0&c^{-1}\end{bmatrix}.
\end{equation*}
A left Haar measure of $H$ is $dh=(ac)^{-1}da\,dc$ and $H$ is unimodular.
The action \eqref{semidirect} is given by $h_{a,c}[y]=a^{-1}y$ and hence $\beta(a,c)=a^{-1/2}c$ and 
$\chi(a,c)=\chi(a)=a$.
The intertwining  is $\Phi(x_1,x_2)=-x_1^2/2$, with Jacobian  $J\Phi(x_1,x_2)=\abs{x_1}$. 
We put $X=\R^2\setminus\{x_1=0\}$ and $Y=(-\infty,0)$.
The representation is
\[
U_{(u;a,c)}f(x_1,x_2)=a^{1/4}c^{-1/2}e^{\pi iux_1^2}f\left(x_1a^{1/2},c^{-1}x_2\right).
\]
The action \eqref{semidirect} is transitive on $Y$, whose origin is
chosen to be $o=-1/2$ with stabilizer given by 
$H_o=\{(1,c):c>0\}$, hence not compact. Therefore we shall use Theorem~\ref{adm-vec1} in order to show that $G$ is reproducing, and then we describe the admissible vectors. Now, the
relevant fiber is $\Phi^{-1}(-1/2)=\{(\pm1,x_2):x_2\in\R\}$. The measure $\nu_o$ is decribed by \eqref{twolines}.
A smooth section $q:Y\to H$ for which $q(y)[-1/2]=y$ is 
\[
q(y)=\begin{bmatrix}(2|y|)^{-1}&0\\0&1\end{bmatrix}.
\] 
Because of Weil's integral 
formula~\eqref{weil}, this forces the  Haar measure of $H_o$  to be $2dc/c$. Next,
\[
S\eta_\pm(y,x_2)=(2|y|)^{-1/4}\eta(\pm\sqrt{2|y|},x_2).
\]
We may now apply Theorem~\ref{adm-vec1}. The right-hand side of \eqref{SECOND} is
\begin{align}
&\int_{\R_-}\int_{\R_+}
\left|\sum_{\pm}\int_\R u_\pm(x_2)(2|y|)^{-1/4}c^{-1/2}\bar\eta\bigl(\pm\sqrt{2|y|},c^{-1}x_2\bigr)\,dx_2\right|^2
\,\frac{2dc}{c}\frac{dy}{2|y|}\nonumber   \\
&\stackrel{x_1=\sqrt{2|y|}}{=}
\int_{\R_+}\int_{\R_+}
\left|\sum_{\pm}\int_\R u_\pm(x_2)c^{-1/2}\bar\eta\bigl(\pm x_1,c^{-1}x_2\bigr)\,dx_2\right|^2
\,\frac{dx_1}{x_1^2}\frac{dc}{c}  \nonumber \\
&=
\int_{\R_+}\int_{\R_+}
\left|\sum_{\pm}\langle u_\pm,\pi_{c}\eta(\pm x_1,\cdot)\rangle\right|^2
\,\frac{dx_1}{x_1^2}\frac{dc}{c},  \label{driftW}
\end{align}
where $\pi_c$ is the non-irreducible representation on $L^2(\R,dt)$ of $\R_+$
\[
\pi_cf(t)=c^{-1/2}f\bigl(c^{-1}t\bigr).
\]
In order to handle \eqref{driftW}, we diagonalize $\pi_c$ with the unitary  Mellin transform. Recall that the unitary Mellin transform is the  operator ${\mathcal M}:L^2(\R_+, dt)\to L^2(\R, d\xi)$ that is obtained by extending
\[
{\mathcal M}f(\xi)=\int_{\R_+}f(t)t^{-2\pi i\xi}\frac{dt}{\sqrt t}
\]
from $L^1(\R_+)\cap L^2(\R_+)$ to $L^2(\R_+)$. Its inverse is defined on $L^1(\R)\cap L^2(\R)$ by
\[
{\mathcal M}^{-1}f(t)=\frac{1}{\sqrt t}\int_{\R}f(\xi)t^{2\pi i\xi}d\xi
\]
and then extended to the whole of  $L^2(\R)$. It is immediate to see that
\[
{\mathcal M}(\pi_cf)(\xi)=c^{-2\pi i \xi}{\mathcal M}f(\xi)
\]
and hence
\begin{align*}
\scal{g}{\pi_c f}&=\scal{{\mathcal M}g}{{\mathcal M}(\pi_c f)}\\
&=\scal{{\mathcal M}g}{c^{-2\pi i \cdot}{\mathcal M} f}\\
&=\sqrt{c}\;\frac{1}{\sqrt c}\int_\R{\mathcal M}g(\xi)c^{2\pi i \xi}\overline{{\mathcal M} f}(\xi)\,d\xi\\
&=\sqrt{c}\,{\mathcal M}^{-1}\left({\mathcal M}g\;\overline{{\mathcal M}f}\,\right)(c).
\end{align*}
Below we use the   following notation: 
$\eta_i(x_1,\cdot)$
is the function on $\R_+$ defined by 
\[t\mapsto\eta_i(x_1,t)=\eta(x_1,(-1)^it)\]
and  $\eta^\sharp_i(x_1,\xi)$ is its Mellin transform evaluated at $\xi$. Take now $u$ supported in the first quadrant $\{(x_1,x_2)\in X:x_1>0, x_2>0\}$ and  denote 
by $u_{I}$ its restriction to the half line $\{(1,t):t>0\}$, a function in $L^2(\R_+,dt)$. Using the above version of Parseval's identity for the Mellin transform, the admissibility conditions given by \eqref{driftW} yield 
\[
\|u\|_{\nu_o}^2=\|u_{I}\|^2_{L^2(\R_+,dt)}
=\int_{\R_+}\Bigl(\int_\R|{\mathcal M}u_{I}(\xi)|^2| \eta_0^\sharp(x_1,\xi)|^2\,d\xi\Bigr)
\,\frac{dx_1}{x_1^2}.
\]
 By the fact that ${\mathcal M}$ is unitary and by Fubini's theorem,  this gives
\[
\int_{\R_+}|\eta_0^\sharp(x_1,\xi)|^2\,\frac{dx_1}{x_1^2}=1,
\qquad\text{a.e. } \xi\in\R.
\]
Similarly one argues for the other three quadrants.  Finally, for an arbitrary function, one uses the above (and the analogous relative to the left-half plane), and deduces that  the mixed terms arising from the square modulus appearing in \eqref{driftW} must vanish. Altogether, the admissibility conditions are therefore 
\begin{equation}
\int_{\R_+}\eta_i^\sharp((-1)^kx_1,\xi)\overline{\eta_j^\sharp((-1)^{\ell}x_1,\xi)}\,\frac{dx_1}{x_1^2}=
\delta_{ij}\delta_{k\ell},
\label{mellin}
\end{equation}
which must  hold for almost every $\xi\in\R$ and all $i,j,k,\ell\in\{0,1\}$. 
Finally, we show how to meet \eqref{mellin}. Take four orthonormal functions $\varphi_\ell\in L^2(\R_+,dx_1)$ satisfying
\[
\int x_1^2|\varphi_\ell(x_1)|^2\,dx_1<+\infty
\]
and one positive function $a\in L^2(\R,d\xi)$, and  define, for $\xi\in\R$
\begin{align*}
\eta_0^\sharp(x_1,\xi)&=\begin{cases} x_1 a(\xi)^{-1/2}\Phi(x_1 a(\xi)^{-1}) & x_1>0\\
- x_1 a(\xi)^{-1/2}\varphi_2(-x_1a(\xi)^{-1}) & x_1<0
\end{cases}\\
\eta_1^\sharp(x_1,\xi)&=\begin{cases} x_1 a(\xi)^{-1/2}\varphi_3(x_1 a(\xi)^{-1}) & x_1>0\\
- x_1 a(\xi)^{-1/2}\varphi_4(-x_1a(\xi)^{-1}) & x_1<0.
\end{cases}\\
\end{align*}
Finally, define
\[
\eta(x_1,x_2)=\begin{cases}(\eta_0^\sharp)^\flat(x_1,x_2 )&x_2>0\\ (\eta_1^\sharp)^\flat(x_1,-x_2)&x_2<0\end{cases}
\]
where we have denoted by $(\eta_i^\sharp)^\flat(x_1,x_2)$ the  inverse Mellin transform evaluated at $x_2$ of the function defined on $\R$ by  $\xi\mapsto\eta_i^\sharp(x_1,\xi)$, for $i=0,1$. We show first that $\eta\in L^2(X,dx)$.
Using the fact that ${\mathcal M}$ is unitary, this is equivalent to proving that
\[
\int_{\R}\int_{\R_+}x_1^2a(\xi)^{-1}|\varphi_\ell(x_1 a(\xi)^{-1})|^2\,dx_1\,d\xi<+\infty,
\]
which follows by the change of variables $\tilde x_1=x_1 a(\xi)^{-1}$ and the hypotheses on $\varphi_\ell$ and on $a$. The admissibility conditions \eqref{mellin} follow from the orthonormality  of the $\varphi_\ell$.

\subsection{Three dimensional groups, $n=2$} The dimension of  $\Sigma$ in the group (3.4) and in the family of groups (3.5) is one, but they are conjugate to semidirect products in ${\mathcal E}_2$ for which the normal factor has dimension $2$. 

\paragraph{The groups(3.4) and (3.5)} Consider the symplectic permutation
\begin{equation*}
w=\begin{bmatrix}0&0&0&-1\\-1&0&0&0\\0&1&0&0\\0&0&-1&0\end{bmatrix}.
\end{equation*}
It is easy to show that it conjugates the group  (3.4), that is $G=\Sigma_3\rtimes \bigl\{\left[\begin{smallmatrix}
e^t & 0 \\
s & e^t \\
\end{smallmatrix}\right]
:s,t\in\R\bigr\}$,
into 
\begin{equation}
wGw^{-1}=
\Sigma_3^\perp\rtimes \bigl\{\left[\begin{smallmatrix}e^t & 0 \\0 & e^{-t} \\\end{smallmatrix}\right]:t\in\R\bigr\}
\label{34}
\end{equation}
and the groups (3.5), that is  $G_\alpha=\Sigma_3\rtimes \bigl\{\left[\begin{smallmatrix}
                                         e^{\alpha t} & 0 \\
                                         s & e^{(\alpha+1)t} \\
                                       \end{smallmatrix}\right]
:s,t\in\R\bigr\}$,
into 
\begin{equation}
wG_\alpha w^{-1}=
\Sigma_3^\perp\rtimes \bigl\{\left[\begin{smallmatrix}e^{-(\alpha+1)t} & 0 \\0 & e^{\alpha t} \\\end{smallmatrix}\right]:t\in\R\bigr\}.
\label{35alpha}
\end{equation}
Recall that the value $\alpha=1/2$ must be neglected because $G_{1/2}$ is unimodular. Next we relabel the groups \eqref{34} and \eqref{35alpha} in a single family. Upon defining
\[
\cos\zeta= -\frac{\alpha+1}{\sqrt{(\alpha+1)^2+\alpha^2}},\qquad
\sin\zeta=\frac\alpha{\sqrt{(\alpha+1)^2+\alpha^2}}
\]
we write
\begin{equation*}
H_\zeta:= \bigl\{h_\zeta(t):=\begin{bmatrix}e^{t\cos\zeta}&\\&e^{t\sin\zeta}\end{bmatrix}:\, t\in \R\bigr\},
\qquad
\zeta\in[-\frac{\pi}{2},\frac{\pi}{2})\setminus\{\zeta_0\},
\end{equation*}
where $\tan\zeta_0=-1/3$ corresponds  to $\alpha=1/2$. Notice that $H_{-\pi/4}$ is conjugate to (3.4).
We show below that all the groups $G_\zeta:=\Sigma_3^\perp\rtimes H_\zeta$ are mutually orbitally equivalent, by exhibiting orbit equivalences of each $G_\zeta$ with $G_0$. 

A left Haar measure on $H_\zeta$ is the pushforward of the Lebesgue measure $dt$ under $t\mapsto h_\zeta(t)$ and $H_\zeta$ is unimodular.
The actions \eqref{semidirect} is 
\[
h_\zeta(t)[y_1,y_2]=
\begin{bmatrix}e^{2t\sin\zeta}&\\&e^{t(\sin\zeta+\cos\zeta)}\end{bmatrix}\begin{bmatrix}y_1\\y_2\end{bmatrix}.
\]
so that $\beta(t)=e^{t(\cos\zeta+\sin\zeta)}$ and $\chi(h_\zeta(t))=e^{-t(3\sin\zeta+\cos\zeta)}$.
The map $\Phi(x)=(-x_2^2/2,-x_1x_2)$ intertwines, and has Jacobian $J\Phi(x)=x_2^2$. Due to the orbit equivalence that we are about to define, we slightly strengthen the natural choice of $X$ and put  $X_\zeta=\R^2\setminus\{x_1x_2=0\}$, and $Y_\zeta=\R_-\times\R\setminus\{0\}$. The representation is
 \[
 U_{(u,v;t)}f(x_1,x_2)=e^{-t(\cos\zeta+\sin\zeta)/2}e^{\pi i(ux_2^2+2vx_1x_2)/2}
 f(h_\zeta(-t)(x_1,x_2)).
 \]
The orbits under the action \eqref{semidirect} all lie  in the left half plane and are: 
open rays issuing from the origin for $\zeta\not\in\{0,-\pi/4\}$; 
open vertical rays issuing from the $y_1$-axis 
and open horizontal rays issuing from the $y_2$-axis in the negative
direction if $\zeta=-\pi/4$, see Fig.~\ref{rays}.

\begin{figure}[htpb]
  \begin{center}
   \subfigure{\includegraphics[width=0.3\linewidth]{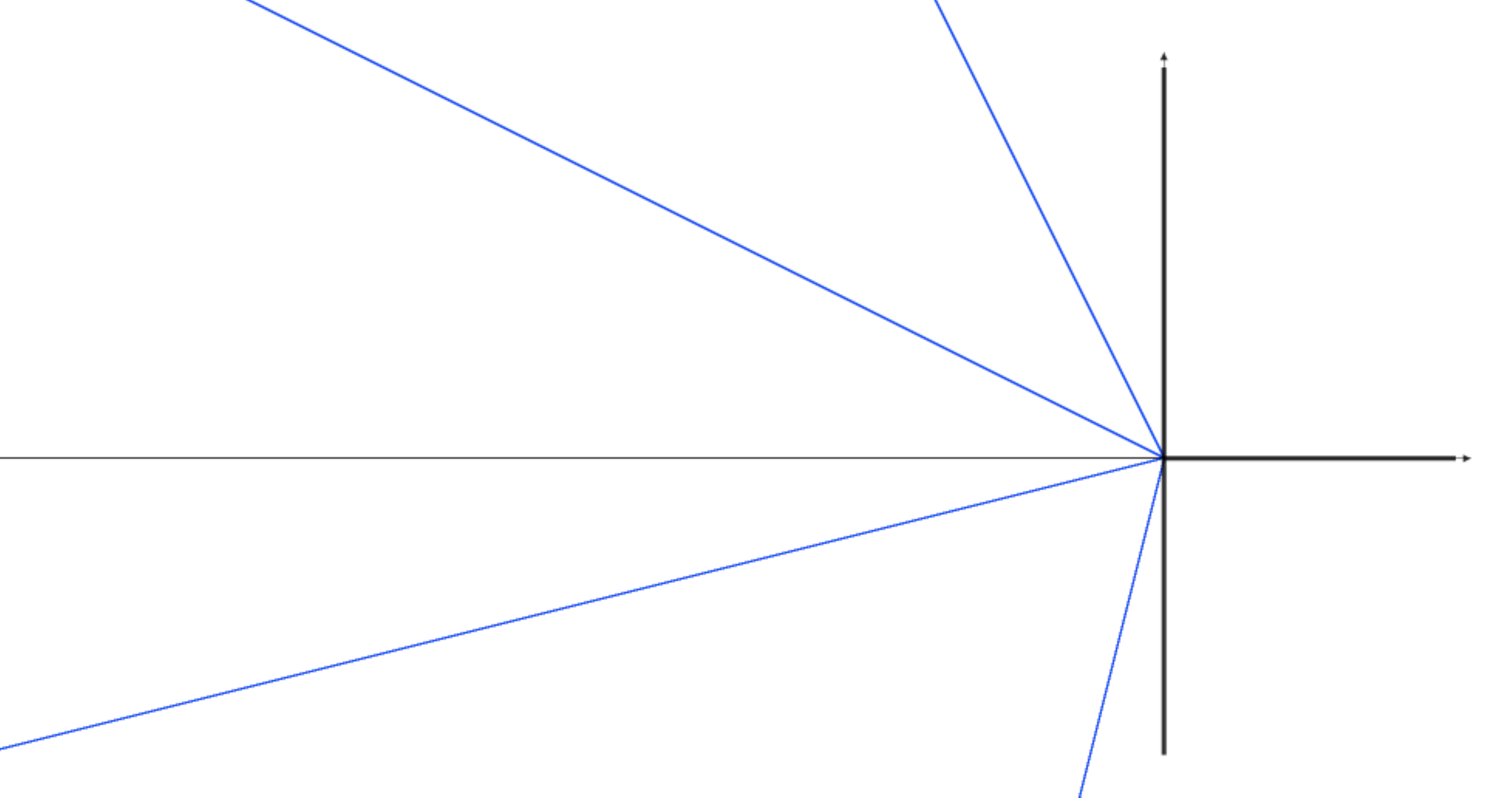} }\quad
    \subfigure{\includegraphics[width=0.3\linewidth]{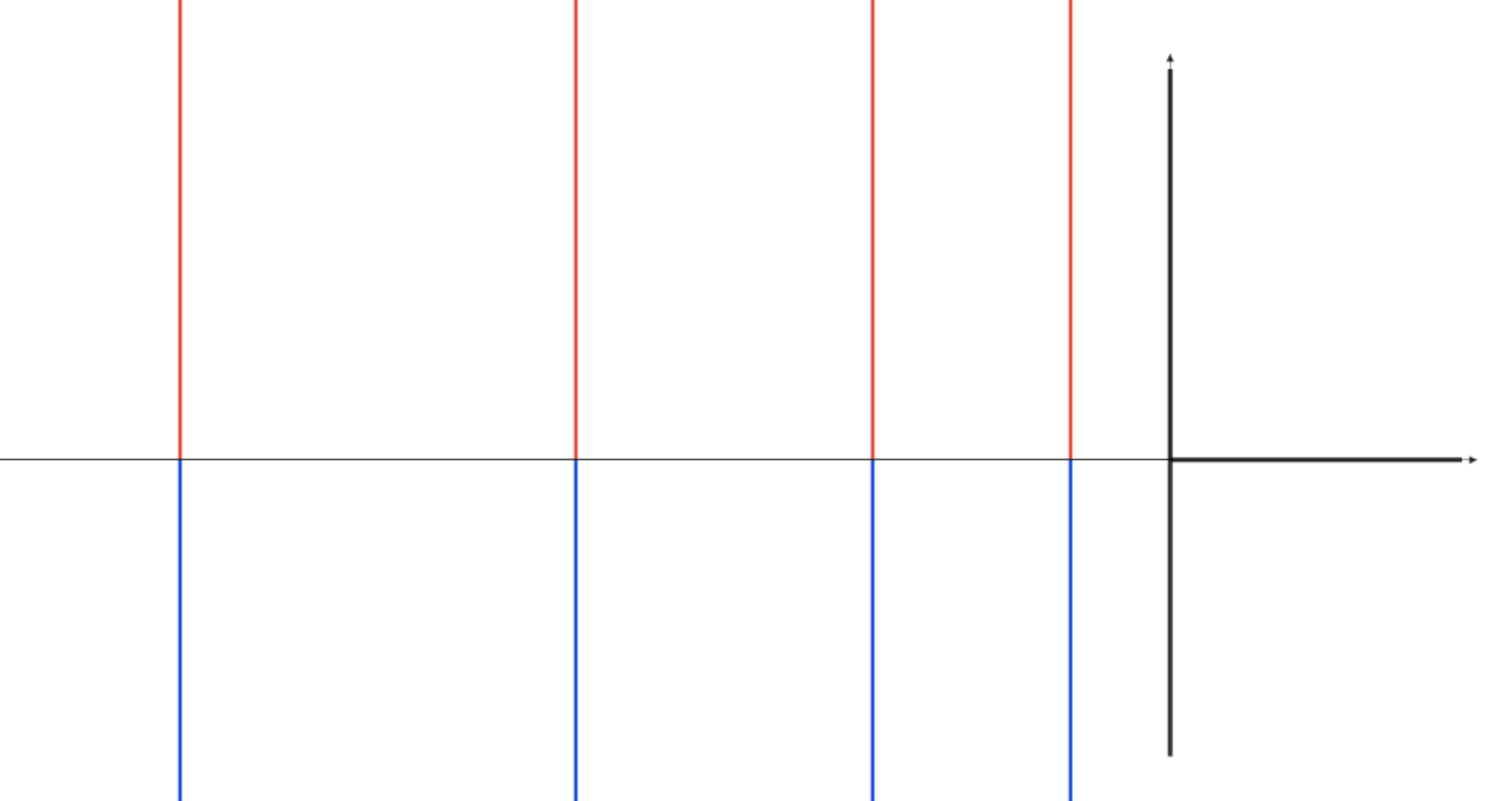}}\quad
    \subfigure{\includegraphics[width=0.3\linewidth]{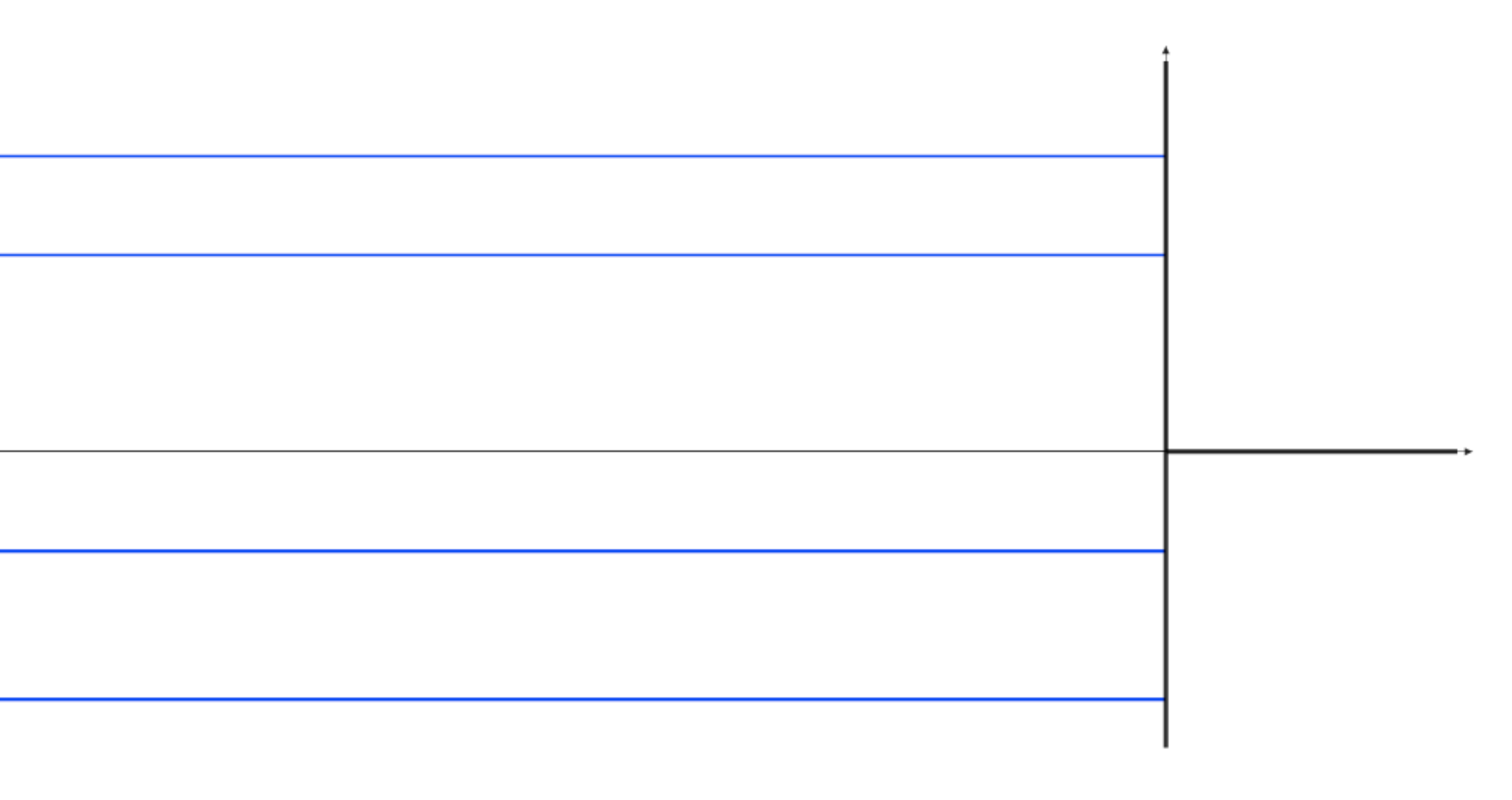} }
    \caption{Left to right: orbits for $\zeta\not\in\{0,-\pi/4\}$, $\zeta=0$ and then
      $\zeta=-\pi/4$.}\label{rays}
  \end{center}

 \end{figure}
 
 \noindent
 Next,  for $\zeta\not=\zeta_0$, we put
\[
c_\zeta=\frac{1}{3\sin\zeta+\cos\zeta}
\]
and define the orbit map from $(G_0, dt, X_0, Y_0)$ to $(G_\zeta, dt, X_\zeta, Y_\zeta)$, with $\zeta\ne0$, as the triple
\begin{align*}
\theta\bigl(\begin{bmatrix}e^t&\\&1\end{bmatrix}\bigr)
&=\begin{bmatrix}e^{ tc_\zeta\cos\zeta}&\\&e^{tc_\zeta\sin\zeta}\end{bmatrix}\\
\psi(x_1,x_2)&=\Bigl( \frac{\sqrt2}{4}\,{\rm sign}(x_1)\,x^2_2\left|x_1x_2\right|^{c_\zeta\cos\zeta},
\sqrt2\,{\rm sign}(x_2)\left|x_1x_2\right|^{c_\zeta\sin\zeta} \Bigr)\\
\xi(y_1,y_2)&=\Bigl(-|y_2|^{2c_\zeta\sin\zeta},
-y_1\,{\rm sign}(y_2)|y_2|^{c_\zeta(\sin\zeta+\cos\zeta)}\Bigr).
\end{align*}
Properties (OM1),(OM2) and (OM3) are easily verified. Notice in particular that the mappings are smooth in the open quadrants in which they are defined. It is tedious but trivial to check that the Jacobian of $\xi$ is equal to $2|c_\zeta\sin\zeta|$, hence constant and non vanishing because $\zeta\ne0$. Thus the triple is an orbit equivalence and we may restrict ourselves to the case $G=G_0$ and for simplicity we write $H$, $X$, $Y$ and so on. For any fixed point $(y_1, y_2)\, \in Y $, the orbit is
\[H[(y_1, y_2)]= \{(y_1, e^{t}y_2)\, : t \in \R\}=\{y_1\}\times \R_\pm\, ,\]
where we take either $\R_+$ or $\R_-$ according to the sign of $y_2$.
As origin of each orbit we chose the intersection of the orbit with the half line $y_2=\pm y_1$, that is
$o(z_1,z_2)=(z_1, {\rm sign}(z_2)|z_1|)$, see Fig.~\ref{orbit}.

            \begin{figure}[H]
 \centering 
             \includegraphics[width=0.40\linewidth]{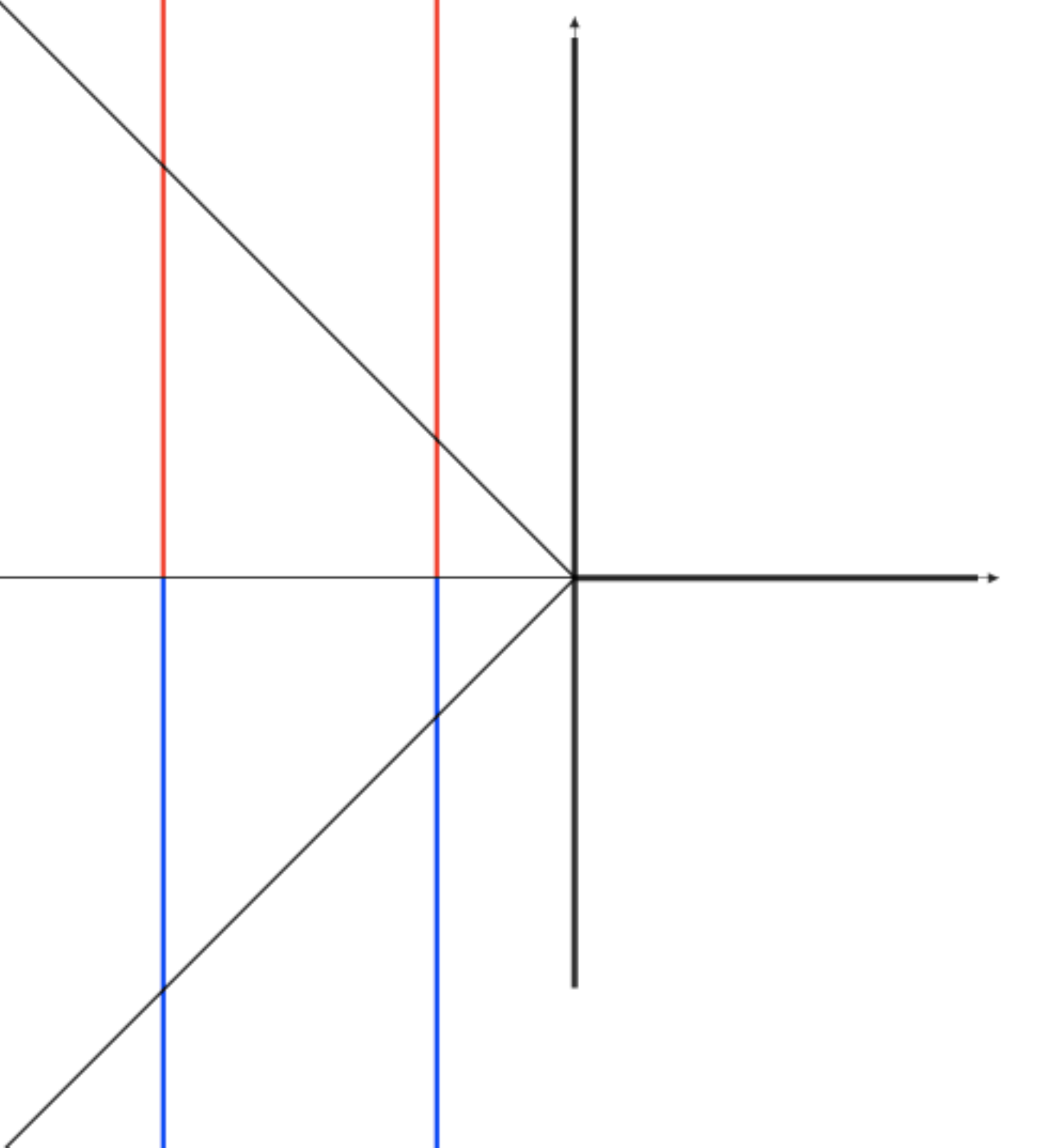}
            \caption{Orbit representatives \label{orbit}}
            \end{figure}

We label the elements of $Z$ with $s\in\R^*$, identifying $s$ with the point $o(s)=(-|s|, s)$ in the plane, namely the origin of the orbit. 
Each stabilizer is the identity of $H$, hence  compact, and the orbits are locally closed.
Therefore the group is reproducing by  Theorem \ref{th:suffnec}.
We shall use Theorem \ref{th:9}  in order to describe the admissible vectors. 

Fix $s\in Z$. An easy computation shows that $\Phi^{-1}(o(s))=\pm(-{\rm sign}(s) \sqrt{|2s|},\sqrt{|2s|})$. The points of the orbit $H[o(s)]$ have the form $y^s_a:=(-|s|, {\rm s}a)$, with $a>0$.
A smooth section $q_s:H[o(s)]\to H$  for which $q_s(y^s_a)[(-|s|, s)]=y^s_a$ is 
\[
q_s(y^s_a)=\begin{bmatrix}\frac a{|s|}&\\&1\end{bmatrix}.
\]
It follows that
\[
S^sf_s(y^s_a,x)=f_{s, a}\left(\frac{|s|x_1}{a}, x_2\right).
\]
The representation $\pi^s$ of the singleton group is the identity on  $L^2(\Phi^{-1}(o(s)))\simeq\C^2$. We fix a basis $\{e_1,e_2\}$ of $\C^2$. If $\eta\in L^2(\R^2)$ and we write $F^j_s:=\scal{S^s\eta_s}{e_j}$, with $j=1,2$, then $\eta$ is admissible if and only if
\[
\int_\R F^j_s(-|s|, {\rm sign}(s) a)\overline{F^\ell_s(-|s|, {\rm sign}(s) a)}\frac{da}a=\delta_{j,\ell},
\qquad j,\ell=1,2.
\]
\paragraph{ (3.6), $\alpha\in[0,+\infty)$} The elements of
$G_\alpha:=\Sigma_1^\perp\rtimes H_\alpha$, with
$H_\alpha=\{h_t:=e^tR_{\alpha t} :t\in\R\}$, are
\begin{equation*}
(u,v;t):=
\begin{bmatrix} e^tR_{\alpha t}&  0 \\
  [\begin{smallmatrix}
    u & v \\ v & -u
  \end{smallmatrix}]e^tR_{\alpha t}
&  e^{-t}R_{\alpha t}\end{bmatrix}.
\end{equation*}
 The Haar measure on $H_\alpha$ is the pushforward
of the Lebesgue measure $dt$ under $t\mapsto h_t$.
The contragradient action~\ref{semidirect} of $H_\alpha$ on $\R$ is $h_t[y]=e^{2t}R_{2\alpha t}\,y$ so that
 $\beta(t)=e^{2t}$ and $\chi(t)=e^{-4t}$. The interwining is
 $\Phi(x)=((x_2^2-x_1^2)/2,-x_1x_2)$, with Jacobian $J\Phi(x)=\|x\|^2$. Hence we set 
 $X_\alpha=\R^2\setminus \{(0,0)\}$ and $Y_\alpha=\R^2\setminus\{(0,0)\}$. The representation is
\[
U_{(u,v;t)}f(x_1,x_2)=e^{-t}e^{\pi  i(u(x_1^2-x_2^2)+ 2 vx_1x_2)}f(e^{-t}R_{-\alpha t}(x_1,x_2)).
\]
We show that these groups are mutually orbitally equivalent. Indeed, fix $\alpha\neq0$ and define
the map $\theta:H_0\to H_\alpha$ by $\theta(e^tI_2)=e^{t}R_{\alpha t}$. We parametrize the elements
in  $X_0$ and $X_\alpha$ with the polar
 coordinates $(r,\theta)$ and define $\psi:X_0\to X_\alpha$ by 
 $\psi(r,\theta)=(r,\theta+\alpha\log (r/\sqrt{2}))$.
 Similarly, we parametrize the elements
in  $Y_0$ and $Y_\alpha$ with the  polar
 coordinates $(\rho,\omega)$ and define $\xi:Y_0\to Y_\alpha$ by 
 $\xi(\rho,\omega)=(\rho,\omega+\alpha\log\rho)$.
 Observe that in polar coordinates $\Phi$ becomes $(r,\theta)\mapsto(r^2/2,2\theta-\pi)$,
 whereas the actions are 
 \begin{equation}
 h_t[(\rho,\omega)]=(e^{2t}\rho,\omega+2\alpha t),
 \qquad
  h_t.(r,\theta)=(e^tr,\theta+\alpha t).
\label{actions36}
\end{equation}
The verification of  (OM1), (OM2)  and (OM3) is now straightforward. A tedious but direct computation shows that $J\xi\equiv1$. Thus we may restrict ourselves to the case $\alpha=0$ and for simplicity we write $H$, $X$, $Y$ and so on, without the index $\alpha=0$.

From \eqref{actions36}, with $\alpha=0$, it is clear that the orbits in $Y$
are rays in $\R^2$ issuing from the origin, hence locally closed.
They are parametrized by the angle $\omega=2\theta+\pi$ with  $\theta\in
(-\pi/2,\pi/2]=:Z$ and $\lambda$ is the Lebesgue measure $d\theta$. For each $\theta$, we pick as origin of the corresponding
orbit the point $o(\theta)=\frac{1}{2}(\cos(2\theta+\pi),\sin(2\theta+\pi))$, whose fiber in $X$ is the pair
$\Phi^{-1}(o(\theta))=\{\pm(\cos\theta,\sin\theta)\}$. The group is reproducing by Theorem~\ref{th:suffnec} because it is non unimodular and the stabilizer is a singleton for every $y\in Y$.

Now we use Theorem \ref{th:corollary 4} to determine the admissible
vectors. A function $\eta\in L^2(\R^2)$ is an admissible vector for
$G_0$ if and only if for almost every $\theta\in (-\pi/2,\pi/2]$
\[
\int_{\R}|\eta(\pm e^{-t} (\cos\theta,\sin\theta))|^2\,e^{2t}\,dt=1,
\qquad
\int_{\R}\eta(e^{-t}(\cos\theta,\sin\theta)\overline{-\eta(e^{-t}(\cos\theta,\sin\theta)}\,e^{2t}\,dt=0.
\]
A final change of variable $r=e^{-t}$ yields
\[
\int_{\R_+}\frac{|\eta(r\cos\theta,r\sin\theta)|^2}{r^3}\,dr=1
\]
for almost every $\theta\in (-\pi,\pi]$, and
\[
\int_{\R_+}\frac{\eta(r\cos\theta,r\sin\theta)\overline{\eta(r\cos(\theta+\pi),r\sin(\theta+\pi))}}{r^3}\,dr=0
\]
 for almost every $\theta\in (-\pi/2,\pi/2]$.
 
\paragraph{ (3.7), $\alpha\in[0,+\infty)$} The elements of
$G_\alpha:=\Sigma_2^\perp\rtimes H_\alpha$, with
$H_\alpha=\{h_t:=e^tA_{\alpha t} :t\in\R\}$, are
\begin{equation}
(u,v;t):=
\begin{bmatrix} e^tA_{\alpha t}&  0 \\
  [\begin{smallmatrix}
    u & v \\ v & u
  \end{smallmatrix}]e^tA_{\alpha t}
&  e^{-t}A_{-\alpha t}\end{bmatrix}.
\label{p2}\end{equation} 
The Haar measure on $H_\alpha$ is the pushforward
of the Lebesgue measure $dt$ under $t\mapsto h_t$.
The contragradient action~\ref{semidirect} of $H_\alpha$ on $\R$ is $h_t[y]=e^{2t}A_{2\alpha t}\,y$ so that
 $\beta(t)=e^{2t}$ and $\chi(t)=e^{-4t}$. The interwining is
 $\Phi(x)=(-(x_1^2+x_2^2)/2,-x_1x_2)$, with Jacobian $J\Phi(x)=|x_1^2-x_2^2|$. Hence we set 
 $X_\alpha=\set{(x_1,x_2)\in\R^2: x_2\neq\pm x_1}$ and 
$Y_\alpha= \set{(y_1,y_2)\in\R^2: y_1<0,\,|y_2|< |y_1|}$. The representation is
\[
U_{(u,v;t)}f(x_1,x_2)=e^{-t}e^{\pi  i(u(x_1^2+x_2^2)+2vx_1x_2)}f(e^{-t}A_{-\alpha t}(x_1,x_2)).
\]
As before these groups are mutually orbitally equivalent. The proof is as in the
previous example by replacing polar coordinates with hyperbolic
ones. Note that  $X$  has now four connected components, each of them
is parametrized by  a pair of hyperbolic variables and is mapped one-to-one
onto $Y$.  We leave the explicit computation to the interested reader. Hence
we shall only treat the case $\alpha=0$.

The orbits in $Y$
are rays in $\R^2$ issuing from the origin with slope between 
$(3/4)\pi$ and  $(5/4)\pi$. Each orbit is  locally
closed and we choose as origin the point $o(\theta)= -\frac12(1,\sin(2\theta))$ with
$\theta\in (-\pi/4,\pi/4)$. Hence $Z=(-\pi/4,\pi/4)$ and $\lambda$ is
the Lebesgue measure $d\theta$. For each $\theta$ the
corresponding fiber $\Phi^{-1}(o(\theta))$ has four points
$\{T_i (\cos\theta,\sin\theta):1\leq i\leq4\}$, where
\[
T_1=I_2,
\quad
T_2=\begin{bmatrix}0&1\\1&0\end{bmatrix},
\quad
T_3=-I_2,
\quad
T_4=-T_2. 
\]
The group
is reproducing by Theorem~\ref{th:suffnec} because it is non
unimodular and all stabilizers are singletons. 

Now we use Theorem \ref{th:corollary 4} to determine the admissible
vectors. A function $\eta\in L^2(\R^2)$ is an admissible vector for
$G_0$ if and only if for almost every $\theta\in (-\pi/4,\pi/4)$ 
\begin{align*}
  & \int_{\R}|\eta(e^{-t}T_i (\cos\theta,\sin\theta))|^2\,e^{2t}\,dt
  =\cos(2\theta),\qquad  1\leq i\leq4\\
  & \int_{\R}\eta(e^{-t}T_i (\cos\theta,\sin\theta))
\overline{\eta(e^{-t}T_j(\cos\theta,\sin\theta)))}\,e^{2t}\,dt=0
  \quad 1\leq i<j\leq 4.
\end{align*}
A final change of variable $r=e^{-t}$ yields
\begin{align*}
&\int_{\R_+}\frac{|\eta(T_i(r\cos\theta,r\sin\theta))|^2}{r^3}\,dr=\abs{\cos(2\theta)},
\qquad1\leq i\leq4\\
&\int_{\R}\frac{\eta(T_i (r\cos\theta,r\sin\theta))
\overline{\eta(T_j (r\cos\theta,r\sin\theta))}}{r^3}\,dr=0
  \quad 0\leq i<j\leq 3
\end{align*}
 for almost every $\theta\in (-\pi/4,\pi/4)$.

\paragraph{ (3.9)} The elements of  $G:=\Sigma_3^\perp\rtimes H$,
with $H=\{h_t:=e^t[
\begin{smallmatrix}
  1 & t \\
 0 & 1
\end{smallmatrix}]
:t\in\R\}$, are 
\begin{equation}
(u,v;t):=
\begin{bmatrix}  e^t & te^t & 0 & 0 \\
  0 & e^t & 0 & 0 \\
 0 &  v e^t &  e^{-t}  & 0    \\
 v e^t & (v+ut)e^t & -t e^{-t} & e^{-t} 
\end{bmatrix}.
\label{p3}\end{equation}
The Haar measure on $H$ is the pushforward
of the Lebesgue measure $dt$ under $t\mapsto h_t$.
The contagradient action~\ref{semidirect}  of $H$ on $\R$ is $h_t[y]=e^{2t}[
\begin{smallmatrix}
  1 & 0 \\
 2t & 1
\end{smallmatrix}]\,y$ so that
 $\beta(t)=e^{2t}$ and $\chi(t)=e^{-4t}$. The intertwining map  is
 $\Phi(x)=(-x_2^2/2,-x_1x_2)$, with Jacobian $J\Phi(x)=x_2^2$. Hence  we define
 $X=\set{(x_1,x_2)\in\R^2: x_2\neq 0}$ and 
$Y= \set{(y_1,y_2)\in\R^2: y_1<0}$. The representation is
\[
U_{(u,v;t)}f(x)=e^{-t}e^{\pi i(u x_2^2+ 2vx_1x_2)}f(e^{-t} (x_1 -tx_2),e^{-t}x_2).
\]
The orbits in $Y$ are curves  whose parametric equation $\gamma_z:\R\to Y$ is
$\gamma_z(t)=- e^t (t,1)\frac{z^2}{2}$ where $z\in (0,+\infty)$
labels the corresponding origin $o(z)=(-\frac{z^2}{2},0)$. Clearly,
the orbits are locally closed, we can set $Z=(0,+\infty)$ and $\lambda$ is
the Lebesgue measure $dz$. For each $z$ the
fiber $\Phi^{-1}(o(z))$ has two points, namely
$(0,\pm z) $. The group
is reproducing by Theorem~\ref{th:suffnec} because it is non
unimodular and all stabilizers are singletons. 

Now we use Theorem \ref{th:corollary 4} to determine the admissible
vectors. A function $\eta\in L^2(\R^2)$ is an admissible vector for
$G$ if and only if for almost every $z\in (0,+\infty)$ 
\begin{align*}
 \int_{\R}|\eta( \pm e^{-t} z(-t,1) )|^2\,e^{2t}\,dt=z^2, \qquad
   \int_{\R}\eta(  e^{-t} z(t,1) )
\overline{\eta(- e^{-t} z(t,1) )}\,e^{2t}\,dt=0. 
\end{align*}
A final change of variable $r=ze^{-t}$ yields
\begin{align*}
  & \int_{\R_+}\dfrac{\left|\eta( \pm (r\log\frac{r}{z},r))
    \right|^2}{r^3}\,dr= 1 \\
   & \int_{\R_+}\frac{ \eta(r\log\frac{r}{z},r)
\overline{\eta( - r\log\frac{r}{z},-r)}}{r^3}\,dr=0 
\end{align*}
 for almost every $z\in (0,+\infty)$.

\subsection{Four dimensional groups, $n=2$}

We  discuss all the details relative to  the groups (4.1) and (4.4), and then give  a quick overview of the other cases. The reason for this choice is that example~(4.1) exhibits new interesting structural features, whereas example~(4.4) is nothing else but the shearlet groups  and is thus a relevant test of our approach.
\paragraph{ (4.1)} The elements of  $G:=\Sigma_3\rtimes H$,
with $H=\{h_{a,b,c}:=[
\begin{smallmatrix}
  a & 0 \\
  b & c
\end{smallmatrix}]
:a,b,c\in\R,\, a>0,c>0\}$ are
\begin{equation}
(u;a,b,c):=
\begin{bmatrix}  a & 0 & 0 & 0 \\
  b & c & 0 & 0 \\
 u a &  0  &  a^{-1}  &  -b a^{-1} c^{-1}    \\
 0 & 0 &   0 & c^{-1} 
\end{bmatrix}.
\label{p4}\end{equation}
The Haar measure on $H$ is the pushforward
of the  measure $a^{-1}c^{-3}dadbdc$ under $(a,b,c)\mapsto h_{a,b,c}$
and the modular function is $\Delta_H(a,b,c)=ac^{-1}$.
The contragradient action of $H$ on $\R$ is $h_{a,b,c}[y]= a^2\,y$ so that
 $\beta(a,b,c)=ac$, $\chi(t)=a^{-2}$. The interwining is
 $\Phi(x)=-x_1^2/2$, with Jacobian $J\Phi(x)=\abs{x_1}$. Hence we set 
 $X=\{(x_1,x_2):x_1\not=0\}$ and $Y=(-\infty,0)$. 
The representation is
\[
U_{(u;a,b,c)}f(x)=(ac)^{-\frac12}e^{\pi iu x_1^2}f(a^{-1}x_1,c^{-1}(x_2-ba^{-1}x_1)).
\]
Clearly, the contragradient action \eqref{semidirect} is transitive on $Y$, whose origin is chosen to be $o=-1/2$ with stabilizer
$H_o=\{(1,b,c):b\in\R,c>0\}$, clearly isomorphic to the affine group, hence not compact.
Therefore we shall use Theorem~\ref{adm-vec1} in order to show that
$G$ is reproducing, and then we describe the admissible vectors. Now,
the relevant fiber is $\Phi^{-1}(o)=\{(\pm1,x_2):x_2\in\R\}$ with measure $\nu_o=\delta_++\delta_-$, where $\delta_\pm$ is the push-forward of one-dimensional Lebesgue measure to the straight line $\{(\pm1,x_2)\}$.
We shall therefore identify $L^2(X,\nu_o)$ with $L^2(\R,dx_2)\oplus L^2(\R,dx_2)$ and write $f_\pm(x_2):=f(\pm1,x_2)$. 

A smooth section $q:Y\to H$ is $q(y)=h_{\sqrt{2|y|},0,1}$. Because of Weil's integral 
formula~\eqref{weil}, this forces the  Haar measure of $H_o$  to be $db\,dc/c^2$. Next,
\[
S\eta_\pm(y,x_2)=(2|y|)^{-1/4}\eta(\pm\sqrt{2|y|},x_2).
\]
We may now apply Theorem~\ref{adm-vec1}. The right-hand side of \eqref{SECOND} is
\begin{align}
&\int_{\R_-}\int_{\R\times\R_+}
\left|\sum_{\pm}\int_\R u_\pm(x_2)(2|y|)^{-1/4}\bar\eta\bigl(\pm\sqrt{2|y|},\frac{x_2\mp b}{c}\bigr)\,dx_2\right|^2
\,\frac{db\,dc}{c^3}\frac{1}{(2|y|)^{3/2}}\,dy\nonumber   \\
&\stackrel{x_1=\sqrt{2|y|}}{=}
\int_{\R_+}\int_{\R\times\R_+}
\left|\sum_{\pm}\int_\R u_\pm(x_2)\bar\eta\bigl(\pm x_1,\frac{x_2\mp b}{c}\bigr)\,dx_2\right|^2
\,\frac{dx_1\,db\,dc}{x_1^3c^3}  \nonumber \\
&=
\int_{\R_+}\int_{\R\times\R_+}
\left|\sum_{\pm}\langle u_\pm,\pi^\pm_{(b,c)}\eta(\pm x_1,\cdot)\rangle\right|^2
\,\frac{dx_1\,db\,dc}{x_1^3c^2}\nonumber \\
&=
\int_{\R_+}\int_{\R\times\R_+}
\left|\langle u_+,\pi^+_{(b,c)}\eta(x_1,\cdot)\rangle
+\langle Tu_-,\pi^+_{(b,c)}T\eta(-x_1,\cdot)\rangle\right|^2
\,\frac{dx_1\,db\,dc}{x_1^3c^2},\label{DW}
\end{align}
where $\pi^{\pm}$ are the standard reproducing representations of the  affine group on $L^2(\R,dx)$
\begin{equation}
\pi^{\pm}_{(b,c)}f(x)=c^{-1/2}f\bigl(\frac{x\mp b}{c}\bigr),
\label{yetanotherwavelet}
\end{equation}
which are mutually equivalent under the intertwining isometry $Tf(x)=f(-x)$. Separating negative and positive  frequencies, the corresponding  orthogonality relations read
\begin{align*}
\int_{\R\times\R_+}\langle u,\pi^+_{(b,c)}v\rangle\langle \pi^+_{(b,c)}v',u'\rangle\,\frac{db\,dc}{c^2}
&=\int_{\R_+}\hat u(\xi)\bar{\hat u}'(\xi)\,d\xi
\int_{\R_+}\frac{\bar{\hat v}(\xi)\hat v'(\xi)}{\xi}\,d\xi\\
&+\int_{\R_+}\hat u(-\xi)\bar{\hat u}'(-\xi)\,d\xi
\int_{\R_+}\frac{\bar{\hat v}(-\xi)\hat v'(-\xi)}{\xi}\,d\xi .
\end{align*}
Upon selecting $u_-=0$, or $u_+=0$, respectively, we get the  following four admissibility conditions
\begin{equation}
\int_{\R_+\times\R_+}\frac{|{\mathcal F}_2\eta\bigl(\pm x_1,\pm\xi_2\bigr)|^2}{x_1^3\,\xi_2}\,dx_1\,d\xi_2=1,
\label{double1}
\end{equation}
where the signs may vary independently. Using these, the mixed terms in \eqref{DW} 
must therefore vanish, and thus yield the remaining two admissibility conditions
\begin{align}
&\int_{\R_+\times\R_+}\frac{{\mathcal F}_2\eta\bigl(x_1,\xi_2\bigr)\overline{{\mathcal F}_2\eta}\bigl(-x_1,-\xi_2\bigr)}{x_1^3\,\xi_2}\,dx_1\,d\xi_2=0\label{double21}\\
&\int_{\R_+\times\R_+}\frac{{\mathcal F}_2\eta\bigl(x_1,-\xi_2\bigr)\overline{{\mathcal F}_2\eta}\bigl(-x_1,\xi_2\bigr)}{x_1^3\,\xi_2}\,dx_1\,d\xi_2=0.\label{double22}
\end{align}
Theorem~\ref{adm-vec1} tells us that $G$ is reproducing if and only if there exist vectors $\eta\in L^2(\R^2)$ satisfying all the above admissibility conditions. It is easy to se that this is the case. First of all, factor 
$\eta(x_1,x_2)=\eta_1(x_1)\eta_2(x_2)$. Then \eqref{double1} becomes
\[
\int_{\R_+}\frac{\eta_1\bigl(\pm x_1\bigr)|^2}{x_1^3}\,dx_1
\int_{\R_+}\frac{|\hat\eta_2\bigl(\pm\xi_2\bigr)|^2}{\xi_2}\,d\xi_2
=1,
\]
and may be met by taking for $\eta_2$ a  usual wavelet, and then by choosing $\eta_1$ to satisfy
\[
\int_{\R_+}\frac{|\eta_1\bigl(\pm x_1\bigr)|^2}{x_1^3}\,dx_1
=1.
\]
As for \eqref{double21} and \eqref{double22}, it is enough to require
\[
\int_{\R_+}\frac{\eta_1\bigl(x_1\bigr)\bar\eta_1\bigl(-x_1\bigr)}{x_1^3}\,dx_1
=0,
\]
which can be achieved with support considerations.


\paragraph{(4.4), $\alpha\in (0,1{]}$ the  shearlet groups}
In order to comply with standard notation, we change some parametrization.
For any $\gamma\leq 1/2$ we set $G_\gamma:=\Sigma_3^\perp\rtimes H_\gamma$, where 
\begin{equation}
\Sigma_3= \bigl\{\left[\begin{smallmatrix}
                                         0 & v/2\\
                            v/2 & u \\
\end{smallmatrix}\right]\bigr\} \qquad
H_\gamma=\bigl\{h_{a,s}:=\left[\begin{smallmatrix}
                                         a^{\frac{1}{2}-\gamma} & -s a^{-1/2} \\
                                         & a^{-\frac{1}{2}} \\
                                       \end{smallmatrix}\right]
:a>0,s\in\R\bigr\} \qquad 
\label{shearmatrices}
\end{equation}
and the action of $H_\gamma$ on the normal abelian factor $\R^2$ is
\begin{equation}
h_{a,s}^\dagger[(u,v)]=\begin{bmatrix}a&sa^\gamma\\0&a^\gamma\end{bmatrix}
\begin{bmatrix}u\\v\end{bmatrix},\label{eq:action-shear}
\end{equation}
which shows that $G_\gamma$ is indeed the shearlet group.

The correspondence with the notation in Theorem~\ref{main} is 
\[
\gamma=(1+2\alpha)/2(\alpha+1) \qquad a=e^{t/(\gamma-1)}
\]
and the standard shearlet group with parabolic scaling corresponds to~$\gamma=1/2$, i.e., $\alpha=0$.  As explained in Part I, we require
that $\gamma\leq1/2$ since the groups with
$\gamma\in(1/2,+\infty)$ are conjugate via a Weyl-group matrix to
those corresponding to $\gamma\in(-\infty,1/2)$.  However, the
description of the admissible vectors that follows holds for all
choices of $\gamma$.

First of all, we show  that there exists no orbit equivalence between any pair of different shearlet groups. We do this by finding all possible isomorphisms $\theta:H_\gamma\to H_{\gamma'}$ and then showing that the condition $\alpha(h)=\alphaÕ(\theta(h))$, which is  necessary in order to have orbit equivalence (see Remark~\ref{samealfa}), implies $\gamma=\gamma'$.   
We start by observing that each of the groups $H_\gamma$ is isomorphic to  the linear group
consisting of the matrices
\[
(a,s)=\begin{bmatrix}a^{1-\gamma}&s\\0&1\end{bmatrix},
\qquad a>0,s\in\R.
\]
They obey the product rule
$(a, s)(a', s')=(aa', s+a^{1-\gamma}sÕ)$, just as those in~ $H_\gamma$,
see~\eqref{eq:action-shear}. Fix now $\gamma\neq\gamma'$ in $(-\infty, 1/2]$.
\begin{lemma}
Any isomorphism  $\theta:H_\gamma\to H_{\gamma'}$ has the form
\[
\theta((a, s))=(a^{\frac{1-\gamma}{1-\gamma'}},ds+b\frac{a^{1-\gamma}-1}{1-\gamma})
\]
for some $b\in \R$ and some $d\neq0$.
\end{lemma}
\begin{proof}
In the course of the proof, we write $\delta=1-\gamma$. By taking derivatives at $t=0$ of the one-parameter subgroups
$\{(e^t,0)\}$ and $\{(1,t)\}$, we recognize at once that  the Lie algebra ${\rm Lie}(H_\gamma)$ has generators
\[
X=\begin{bmatrix}\delta&0\\0&0\end{bmatrix},
\qquad
Y=\begin{bmatrix}0&1\\0&0\end{bmatrix},
\]
and their  bracket is $[X,Y]=\delta$.  Similarly,  the Lie algebra ${\rm Lie}(H_{\gamma'})$ has generators
given by 
$X'=\left[\begin{smallmatrix}\delta'&0\\0&0\end{smallmatrix}\right]$ and $Y'=Y$ with bracket $[X',Y']=\delta'$. 
Let $d\theta$ denote the differential of a Lie group isomorphism $\theta:H_\gamma\to H_{\gamma'}$, and  denote by $\left[\begin{smallmatrix}a&b\\c&d\end{smallmatrix}\right]$ the matrix representing $d\theta$ in the bases $\{X,Y\}$ and $\{X',Y'\}$. Thus $ad-bc\neq0$ and  since $d\theta$ is a Lie algebra homomorphism
\[
\delta(cX'+dYÕ)=\delta d\theta(Y)=d\theta([X,Y])=[aX'+bYÕ,cX'+dYÕ]=(ad-bc)\delta'YÕ.
\]
This forces $c=0$ and $\delta=a\delta'$. Therefore
\[
d\theta=\begin{bmatrix}\delta/\delta'&b\\0&d\end{bmatrix},
\qquad
d\neq0, b\in\R.
\]
We thus have
\begin{align*}
\theta(\exp sY)&=\theta(\begin{bmatrix}1&s\\0&1\end{bmatrix})
=\exp(s\,d\theta (Y))=\exp(sdYÕ)=\begin{bmatrix}1&sd\\0&1\end{bmatrix}\\
\theta(\exp tX)&=\theta(\begin{bmatrix}e^{t\delta}&0\\0&1\end{bmatrix})
=\exp(t\,d\theta (X))=\exp(t\begin{bmatrix}(\delta/\delta')\delta&b\\0&1\end{bmatrix}))
=\begin{bmatrix}(e^{t\delta})^{\delta/\delta'}&b\left(\frac{e^{t\delta}-1}{\delta}\right)\\0&1\end{bmatrix}.
\end{align*}
This means that $\theta((1,s))=(1,ds)$ and, with $a=e^t$,
$\theta((a,0))=(b(a^\delta-1)/\delta,a^{\delta/\delta'})$.
The result follows from $\theta((a,s))=\theta((1,s))\theta((0,a))$.
\end{proof}
We now come back to the issue of orbit equivalence.We know  that ,
under the correspondence $(a,s)\leftrightarrow h$ we have
$\alpha(h)=a^{\gamma+1}=a^{2-\delta}$.  Since
$\alpha(h)=\alpha'(\theta(h))$ is  assumed to be valid for every $h\in
H_\gamma$, then 
$a^{2-\delta}=(a^{\delta/\delta'})^{2-\delta'}$ for all $a>0$, and this implies $\delta=\delta'$.
\vskip0.1truecm

Our next goal is to prove that $G_\gamma$ is reproducing.
The Haar measure on $H_\gamma$ is the pushforward
of the measure $a^{\gamma-2}dads$ under $(a,s)\mapsto h_{a,s}$.
The contragradient action of $H_\gamma$ on $\R^2$ is 
\begin{equation}\label{eq:actionshear}
    h_{a,s}[(y_1,y_2)]=\begin{bmatrix}a^{-1}&0\\-sa^{-1}&a^{-\gamma}\end{bmatrix}
\begin{bmatrix}y_1\\ y_2\end{bmatrix}=
\begin{bmatrix}a^{-1}\,y_1\\ a^{-\gamma}(y_2-y_1sa^{\gamma-1})\end{bmatrix}
\end{equation}
so that  $\beta(a,s)=a^{-\gamma}$, $\chi(a,s)=a^{(\gamma+1)}$. The intertwining map  is defined on the whole of  $\R^2 $ by 
 $\Phi(x_1,x_2)=-\frac{1}{2}(x_2^2,x_1x_2)$ and has Jacobian
 $J(\Phi)(x)=x_2^2/2$. We set $X=\{x\in\R^2:x_2\neq 0\}$
 and $Y=\{(y_1,y_2):y_1< 0,y_2\in\R\}$. The representation is
\[
U_{(u,v;a,s)}f(x_1,x_2)= a^{-\frac{\gamma}{2}}   e^{\pi i(u x_2^2+ vx_1x_2)}f(
h_{a,s}^{-1}(x_1,x_2)).
\]
By~\eqref{eq:actionshear} the action~\eqref{eq:actionshear} on $Y$ is transitive and the
stability subgroups are singleton.   The group
is reproducing by Theorem~\ref{th:suffnec}. 

Our next goal is to compute the admissible vectors. In order to illustrate the two possible approaches arising (in this case) from the general theory, we  shall perform the computation both ways:
via Theorem \ref{th:corollary 4} and via Theorem~\ref{th:9}. 

\vskip0.1truecm
{\bf First method.}
Clearly  $Z=\set{0}$ and $\lambda$ is $\delta_0$.  We select the
origin  $o=(-1/2,0)$. Evidently, $\Phi^{-1}(o)=\{\pm
(0,1)\}$. Condition $\eqref{eq:corollary 4}$  becomes
\[
\frac{1}{2}=\int_{\R^2_r}\abs{\eta(\pm\,h_{a,s}^{-1}\,(0,1))}^2\,a^{\gamma-3}\,dads=
\int_{\R^2_r}\abs{\eta(\pm\,(sa^{\gamma-1/2},a^{1/2}) )}^2\,\,a^{\gamma-3}\,dads.
\]
For simplicity, we introduce the maps $\Gamma^\pm\colon\R^2_r\to\R^2_\pm$ defined by $\Gamma^\pm(a,s)=\pm\,(sa^{\gamma-1/2},a^{1/2})$, whose Jacobians are $J(\Gamma^\pm)(a,s)=a^{\gamma-1}/2$. Here  $\R^2_\pm=\{(x_1,x_2)\in\R^2:\pm x_2>0\}$. The above relation can thus be rewritten as
\[
\frac{1}{2}
=\int_{\R^2_r}\abs{\eta(\Gamma^\pm(a,s) )}^2\;\frac{2J(\Gamma^\pm)(a,s)}{(a^{1/2})^4}\,dads
=2\int_{\R^2_\pm}\frac{\abs{\eta( x_1,x_2 )}^2}{x_2^4}\,dx_1dx_2.
\]
Similarly, condition $\eqref{eq:corollary 4}$ for the pair $(P_1,P_2)$ yields
\[
0=\int_{\R^2_+}\frac{\eta(x_1,x_2)\overline{\eta(-x_1,-x_2)}}{x_2^4}\,dx.
\]
In conclusion, a function $\eta\in L^2(\R^2)$ is an admissible vector for $G_\gamma$ if and only if
\begin{equation}
\int_{\R^2_\pm}\frac{\abs{\eta( x_1,x_2 )}^2}{x_2^4}\,dx_1dx_2=\frac{1}{4}\,,\qquad
\int_{\R^2_+}\frac{\eta(x_1,x_2)\overline{\eta(-x_1,-x_2)}}{x_2^4}\,dx_1dx_2=0.
\label{shearlets}
\end{equation}

{\bf Second method.}
Next, we re-obtain the above conditions by applying Theorem~\ref{th:9}, that is, by decomposing  the restriction of the metaplectic representation to $G_\gamma$ into its irreducible components. 
The first step of this process consists in applying Proposition \ref{prop:dis1}, thereby obtaining the disintegration formula $dx=\int_Y \nu_y\,dy$ for the measure $dx$ on $X$. For $y\in Y$ we have $\Phi^{-1}(y)=\{P_1^y,P_2^y\}$ with $P_1^y=\sqrt2(y_2/\sqrt{|y_1|},\sqrt{|y_1|})$ and $P_2^y=-P_1^y$. To compute the measures $\nu_y$ we use  (A.9) in \cite{dede10} recalling that $\nu_y$ is concentrated on $\{P_1^y,P_2^y\}$, and we obtain
\[
\nu_y(P_i^y)=\frac{1}{J\Phi(P_i^y)}=\frac{1}{|y_1|}.
\]
In particular, for each of the inverse images $P^o_i$ of the origin $o$ we have  $\nu_0(P_i)=2$.
As in Section~\ref{sub:analytic} we write $f=\int_Y f_y\,dy$ for every $f\in L^2(X,dx)=\int_Y L^2(X,\nu_y)\,dy$.
Since $Z=\{0\}$ the second disintegration discussed in Proposition~\ref{prop:dis2} is trivial and we simply get $d\tau=dy$.

The second step is to write the operator $S$ that intertwines the metaplectic representation with the relevant induced representation.  An easy computation shows that the map
\[
q\colon Y\to H_\gamma, \quad q(y)=h_{\frac{1}{2|y_1|},\frac{y_2}{|y_1|}}
\]
satisfies $q(y)[o]=y$ and so $q(o)=I$. The map $q$ gives a diffeomorphism between $Y$ and $H_\gamma$ and has Jacobian $Jq(y)=1/2|y_1|^{3}$. By \eqref{weil} we obtain for $\varphi\in C_c(H)$
\begin{align*}
\int_Y \varphi (q(y))\operatorname{vol}(H_o) \,dy
&=\int_{\R^2_r} \varphi(h_{a,s})\alpha(h_{a,s})^{-1}a^{\gamma-2}\,dads \\
&=\int_{\R^2_r} \varphi(h_{a,s})a^{-3}\,dads \\
&=\int_{Y} \varphi(q(y))8|y_1|^3Jq(y)\,dy\\
&=4\int_{Y} \varphi(q(y))\,dy,
\end{align*}
whence $\operatorname{vol}(H_o)=4$. Since $\nu_o$ is concentrated on $\{P_1^o,P_2^o\}$, an orthonormal basis for $L^2(X,\nu_o)$ is $\{\chi_1/\sqrt2,\chi_2/\sqrt2\}$ where $\chi_i$ is the indicator function of $P_i^o$.
Hence the components of the map $S\colon L^2(X,dx)\to L^2(Y,dy;L^2(X,\nu_o))$ defined in \eqref{eq:Sz2} are
\[
(S\eta)_i(y)=\sqrt{\alpha(q(y))\beta(q(y))}\eta_y(q(y).P_i^o)=(2|y_1|)^{-1/2}\eta_y(P_i^y).
\]
where  $\eta=\int_Y \eta_y\,dy \in L^2(X,dx)$.
We now use the decomposition discussed in Section~\ref{sub:compact}. Since the stabilizer $H_o=\{I\}$ is trivial, the representation $\pi$ of $H_o$ on $L^2(X,\nu_0)$  decomposes as
\[
\pi=2 \,{\rm id}_\C={\rm id}_\C\oplus{\rm id}_\C,\qquad L^2(X,\nu_0)=2\,\C=\C\oplus\C.
\]
Looking at \eqref{eq:decomp_Sf}, we have $F^i=\sqrt{2}(S\eta)_i$ because $\{\chi_1/\sqrt2,\chi_2/\sqrt2\}$ is an orthonormal basis.
Finally, we are in a position to apply Theorem \ref{th:9}. Take $\eta\in L^2(X,dx)$ and write $F^i$ as above. Relation \eqref{eq:th9} becomes
\[
\begin{split}
\frac{\delta_{ij}}{4}&= \intl_Y \scal{F^{i}(y)}{F^{j}(y)} \frac{\alpha(q(y))}{\Delta_H(q(y))}\,dy\\
&=2\intl_Y \scal{(S\eta)_i(y)}{(S\eta)_j(y)} |2y_1|^{-2}\,dy\\
&=2\intl_Y \scal{\eta_y(P_i^y)}{\eta_y(P_j^y)} |2y_1|^{-3}\,dy.
\end{split}
\]
If $i=j$, with the change of variable $y=\Phi(x)$ we get 
\[
\frac{1}{4}=2\intl_{\R^2_\pm} \abs{\eta( x_1,x_2)}^2 x_2^{-6}J\Phi( x_1,x_2)\,dx_1dx_2
=\int_{\R^2_\pm}\frac{\abs{\eta( x_1,x_2 )}^2}{x_2^4}\,dx_1dx_2.
\]
The orthogonality condition in \eqref{shearlets} is obtained similarly.

\paragraph{ (4.2)} The elements of $G=\Sigma_1^\perp\rtimes
\{h_{t,\theta}:=e^tR_\theta :t\in\R,\,\theta\in [0,2\pi)\}$ are
\begin{equation*}
(u,v;t):=
\begin{bmatrix} e^tR_{\theta}&  0 \\
  [\begin{smallmatrix}
    u & v \\ v & -u
  \end{smallmatrix}]e^tR_{\theta}
&  e^{-t}R_{\theta}\end{bmatrix}.
\end{equation*}
The Haar
measure on $H$ is the pushforward
of the  measure $dtd\theta/2\pi$ under $(t,\theta)\mapsto h_{t,\theta}$
and $H$ is unimodular.
The contragradient action of $H$ on $\R^2$ is $h_{t,\theta}[y]=
e^{2t}R_{2\theta}\,y$ so that
$\beta(t,\theta)=e^{2t}$ and $\chi(t,\theta)= e^{-4t}$. The interwining is
$\Phi(x)=((x_2^2-x_1^2)/2,-x_1x_2)$ with Jacobian $J\Phi(x)=\|x\|^2$.
Hence we set
$X=\R^2\setminus \{(0,0)\}$ and $Y_\alpha=\R^2\setminus\{(0,0)\}$.
The representation is
\[
U_{(u,v;t,\theta)}f(x_1,x_2)=e^{-t}e^{\pi  i(u(x_1^2-x_2^2)+ 2
vx_1x_2)}f(e^{-t}R_{-\theta}(x_1,x_2)).
\]
The action \eqref{semidirect} is transitive on
$Y$. We  choose as origin $o=-(1/2,0)$, whose compact stabilizer is
$H_o=\{I,R_\pi\}$ and whose fiber in $X$ is the pair
$\Phi^{-1}(o)=\{\pm(1,0)\}$.
The group is reproducing by Theorem \ref{th:suffnec}
and we use Theorem \ref{th:corollary 4} to determine the admissible
vectors. Since $-(1,0)=R_\pi (1,0)$,
a function $\eta\in L^2(\R^2)$ is an admissible vector for
$G$ if and only if
\begin{align*}
& \int_{\R\times [0,+\infty)} |\eta(e^{-t} (\cos\theta,-\sin\theta))|^2 e^{2t}
\frac{dtd\theta}{2\pi}= 1 \\
& \int_{\R\times [0,+\infty)} \eta(e^{-t}(\cos\theta,-\sin\theta))
\overline{\eta}(-e^{-t}R_{-\theta} (1,0))  e^{2t}
\frac{dtd\theta}{2}= 0
\end{align*}
The change of variables
$(x_1,x_2)=e^{-t}(\cos\theta,-\sin\theta)$, whose Jacobian is $e^{-2t}$,
yields
\begin{align*}
&\int_{\R^2}  |\eta(x)|^2 \, \frac{dx}{|x|^4}= 2\pi \\
& \int_{\R^2}  \eta(x)
\overline{\eta}(-x)
 \, \frac{dx}{|x|^4}= 0.
\end{align*}

\paragraph{ (4.3)} The elements of
$G=\Sigma_2^\perp\rtimes \{h_{t,s}:=e^tA_s :t,s\in\R\}$ are
\begin{equation*}
(u,v;t):=
\begin{bmatrix} e^tA_{st}&  0 \\
  [\begin{smallmatrix}
    u & v \\ v & u
  \end{smallmatrix}]e^tA_{s}
&  e^{-t}A_{s}\end{bmatrix}.
\end{equation*} 
The Haar
measure on $H$ is the pushforward of the measure
$dtds$ under $(t,s)\mapsto h_{t,s}$ and $H$ is unimodular.
The  action~\eqref{semidirect} of $H$ on $\R^2$ is
$h_{t,s}[y]=e^{2t}A_{2s}\,y$ so that
$\beta(t,s)=e^{2t}$ and $\chi(t,s)=e^{-4t}$. The interwining is
$\Phi(x)=(-(x_1^2+x_2^2)/2,-x_1x_2)$, with Jacobian
$J\Phi(x)=|x_1^2-x_2^2|$. 

We set
$X=\set{(x_1,x_2)\in\R^2: x_2\neq\pm x_1}$ and
$Y= \set{(y_1,y_2)\in\R^2: y_1<0,\,|y_2|< |y_1|}$, and the representation is
\[
U_{(u,v;t,s)}f(x_1,x_2)=e^{-t}e^{\pi  i(u(x_1^2+x_2^2)+2vx_1x_2)}f(e^{-t}A_{-s}(x_1,x_2)).
\]
Clearly, the  action \eqref{semidirect} is transitive
and free on
$Y$, whose origin is  $o=-(1/2,0)$ and the corresponding
fiber in $X$ is
$\Phi^{-1}(o)=\{M_i(1,0):i=0,\ldots,3\}$ where
\[
M_0=I,\quad M_1=\Omega,\quad M_2=-I,\quad M_3=-\Omega
\]
must be seen as isomorphisms of $\R^2$.

The group is reproducing by Theorem \ref{th:suffnec}
and we use Theorem \ref{th:corollary 4} to determine the admissible
vectors. Set $C=\set{(x_1,x_2): x_1>0, |x_2|<x_1}$ and observe
 that the maps $(t,s)\mapsto (x_1,x_2)=e^{-t}A_sM_i(1,0)$ are  diffeomorphisms
from $\R^2$
 onto $M_iC$  with Jacobian $e^{-2t}$. Hence, reasoning as in the
 previous example, a function $\eta\in L^2(\R^2)$ is an admissible vector for
$G$ if and only if
\begin{align*}
&\int_{M_iC}  |\eta(x)|^2 \, \frac{dx}{|x_1^2-x_2^2|^2}= 1 \qquad i=0,\ldots,3\\
& \int_{C}  \eta(M_ix)
\overline{\eta}(M_jx)
 \, \frac{dx}{|x_1^2-x_2^2|^2}= 0\qquad 0\leq i<j\leq 3,
\end{align*}
where in the second set of equations we use that $M_i$ are unimodular
maps of $C$ onto $M_iC$.

\end{document}